\documentclass[a4paper, 11pt]{amsart}   
\usepackage{mathptmx, amssymb,amscd,latexsym, eulervm}   
\usepackage{amsmath}
\usepackage{amsthm}
\usepackage{mathdots}
\usepackage[pagebackref,colorlinks=true,linkcolor=blue,urlcolor=blue]{hyperref}
\usepackage{color}
\usepackage[onehalfspacing]{setspace}
\usepackage{tabularx}
\usepackage{amsfonts}
\usepackage{paralist}
\usepackage{aliascnt}
\usepackage[initials, lite]{amsrefs}
\usepackage{amscd}
\usepackage{blkarray}
\usepackage{mathbbol}
\usepackage{setspace}
\usepackage[inner=2.4cm,outer=2.4cm, bottom=3.2cm]{geometry}
\usepackage{tikz, tikz-cd}
\usepackage{calligra,mathrsfs}

\usepackage{tikz}
\usetikzlibrary{matrix}
\usetikzlibrary{arrows,calc}
\allowdisplaybreaks

\BibSpec{collection.article}{%
	+{}  {\PrintAuthors}                {author}
	+{,} { \textit}                     {title}
	+{.} { }                            {part}
	+{:} { \textit}                     {subtitle}
	+{,} { \PrintContributions}         {contribution}
	+{,} { \PrintConference}            {conference}
	+{}  {\PrintBook}                   {book}
	+{,} { }                            {booktitle}
	+{,} { }                            {series}
	+{, vol.} { }                            {volume}
	+{,} { }                            {publisher}
	+{,} { \PrintDateB}                 {date}
	+{,} { pp.~}                        {pages}
	+{,} { }                            {status}
	+{,} { \PrintDOI}                   {doi}
	+{,} { available at \eprint}        {eprint}
	+{}  { \parenthesize}               {language}
	+{}  { \PrintTranslation}           {translation}
	+{;} { \PrintReprint}               {reprint}
	+{.} { }                            {note}
	+{.} {}                             {transition}
	+{}  {\SentenceSpace \PrintReviews} {review}
}
\AtBeginDocument{%
	\def\MR#1{}
}

\makeatletter
\@namedef{subjclassname@2020}{%
	\textup{2020} Mathematics Subject Classification}
\makeatother

\newcommand{\kk}{\mathbb{k}}

\newcommand{\bu}{\mathbf{u}}

\newcommand{\A}{\mathcal{A}}

\newcommand{\Pl}{\mathscr{P}_{\ls }}

\newcommand{\NN}{\normalfont\mathbb{N}}
\newcommand{\ZZ}{\normalfont\mathbb{Z}}
\newcommand{\gin}{{\normalfont\text{gin}}}

\newcommand{\PP}{{\normalfont\mathbb{P}}}

\newcommand{\mm}{{\normalfont\mathfrak{m}}}
\newcommand{\K}{{\mathcal{K}}}
\newcommand{\init}{{\normalfont\text{in}}}

\newcommand{\bbS}{\mathfrak{J}}
\newcommand{\QQ}{\mathbb{Q}}
\newcommand{\pp}{{\normalfont\mathfrak{p}}}

\newcommand{\nn}{{\normalfont\mathfrak{N}}}
\newcommand{\bn}{{\normalfont\mathbf{n}}}
\newcommand{\bm}{{\normalfont\mathbf{m}}}
\newcommand{\supp}{{\normalfont\text{supp}}}
\newcommand{\Msupp}{{\normalfont\text{MSupp}}}
\newcommand{\bx}{{\normalfont\mathbf{x}}}

\newcommand{\ttt}{{\normalfont\mathbf{t}}}
\newcommand{\www}{{\normalfont\mathbf{w}}}
\newcommand{\zzz}{{\normalfont\mathbf{z}}}

\newcommand{\nnn}{{\normalfont\mathfrak{n}}}
\newcommand{\fP}{{\mathfrak{P}}}
\newcommand{\fJ}{{\mathfrak{J}}}

\newcommand{\reg}{\normalfont\text{reg}}
\newcommand{\ttop}{ {\normalfont\text{top}} }

\newcommand{\codim}{{\normalfont\text{codim}}}

\newcommand{\Ker}{\normalfont\text{Ker}}

\newcommand{\Quot}{\normalfont\text{Quot}}

\newcommand{\Sym}{\normalfont\text{Sym}}
\newcommand{\Rees}{\mathcal{R}}

\newcommand{\ee}{{\normalfont\mathbf{e}}}

\newcommand{\OO}{\mathcal{O}}

\newcommand{\LL}{\mathbb{L}}

\newcommand{\C}{\mathscr{C}}

\newcommand{\HH}{\normalfont\text{H}}

\newcommand{\St}{\normalfont\text{St}}

\newcommand{\Hilb}{{\normalfont\text{Hilb}}}
\newcommand{\Spec}{\normalfont\text{Spec}}

\newcommand{\multProj}{\normalfont\text{MultiProj}}
\newcommand{\msupp}{\normalfont\text{MSupp}}
\newcommand{\hsupp}{\normalfont\text{HSupp}}

\newcommand{\exc}{\hyperref[exch]{Exchange} }
\newcommand{\expa}{\hyperref[expa]{Expansion} } 
\newcommand{\excs}{\hyperref[exch]{Exchange} }

\def\fu{\mathbf{u}}
\def\f0{\mathbf{0}}

\def\fe{\mathbf{e}}
\def\fw{\mathbf{w}}
\def\fv{\mathbf{v}}

\def\fs{\mathbf{s}}
\def\bb{\mathbf{b}}

\def\fa{\mathbf{a}}

\def\fn{\mathbf{n}}
\def\1{\mathbf{1}}
\def\ls{\leqslant}
\def\gs{\geqslant}

\def\ba{\mathbf{a}}
\def\bc{\mathbf{c}}
\def\be{\mathbf{e}}
\def\bw{\mathbf{w}}
\def\bv{\mathbf{v}}
\def\bx{\mathbf{x}}

\def\bb{\mathbf{b}}
\def\bc{\mathbf{c}}

\def\ba{\mathbf{a}}

\newtheorem{theorem}{Theorem}[section]

\newtheorem{headthm}{Theorem}

\newaliascnt{headcor}{headthm}

\aliascntresetthe{headcor}

\newaliascnt{headconj}{headthm}

\aliascntresetthe{headconj}

\newaliascnt{corollary}{theorem}
\newtheorem{corollary}[corollary]{Corollary}
\aliascntresetthe{corollary}

\newaliascnt{claim}{theorem}

\aliascntresetthe{claim}

\newaliascnt{lemma}{theorem}
\newtheorem{lemma}[lemma]{Lemma}
\aliascntresetthe{lemma}

\newaliascnt{conjecture}{theorem}
\newtheorem{conjecture}[conjecture]{Conjecture}
\aliascntresetthe{conjecture}

\newaliascnt{proposition}{theorem}
\newtheorem{proposition}[proposition]{Proposition}
\aliascntresetthe{proposition}

\theoremstyle{definition}
\newaliascnt{definition}{theorem}
\newtheorem{definition}[definition]{Definition}
\aliascntresetthe{definition}

\newaliascnt{notation}{theorem}
\newtheorem{notation}[notation]{Notation}
\aliascntresetthe{notation}

\newaliascnt{example}{theorem}
\newtheorem{example}[example]{Example}
\aliascntresetthe{example}

\newaliascnt{examples}{theorem}

\aliascntresetthe{examples}

\newaliascnt{remark}{theorem}
\newtheorem{remark}[remark]{Remark}
\aliascntresetthe{remark}

\newaliascnt{fact}{theorem}
\newtheorem{fact}[fact]{Fact}
\aliascntresetthe{fact}

\newaliascnt{question}{theorem}
\newtheorem{question}[question]{Question}
\aliascntresetthe{question}

\newaliascnt{questions}{theorem}

\aliascntresetthe{questions}

\newaliascnt{problem}{theorem}

\aliascntresetthe{problem}

\newaliascnt{construction}{theorem}

\aliascntresetthe{construction}

\newaliascnt{setup}{theorem}
\newtheorem{setup}[setup]{Setup}
\aliascntresetthe{setup}

\newaliascnt{algorithm}{theorem}

\aliascntresetthe{algorithm}

\newaliascnt{observation}{theorem}

\aliascntresetthe{observation}

\newaliascnt{defprop}{theorem}

\aliascntresetthe{defprop}

\def\equationautorefname~#1\null{(#1)\null}
\def\sectionautorefname~#1\null{Section #1\null}
\def\subsectionautorefname~#1\null{\S #1\null}

\def\trdeg{{\rm trdeg}}

\title{$K$-polynomials of multiplicity-free varieties}

\author{Federico Castillo}
\address[Castillo]{Departamento de Matem\'aticas, Pontificia Universidad Cat\'olica de Chile, Santiago, Chile}
\email{federico.castillo@mat.uc.cl}

\author{Yairon Cid-Ruiz}
\address[Cid-Ruiz]{Department of Mathematics, North Carolina State University, Raleigh, NC 27695, USA}
\email{ycidrui@ncsu.edu}

\author{Fatemeh Mohammadi}
\address[Mohammadi]{Department of Mathematics, KU Leuven, Celestijnenlaan 200B, Leuven, Belgium, and Department of Computer Science, KU Leuven, Celestijnenlaan 200A, 3001 Leuven, Belgium, and
	Department of Mathematics and Statistics,
	UiT – The Arctic University of Norway, 9037 Troms\o, Norway}
\email{fatemeh.mohammadi@kuleuven.be}

\author{Jonathan Monta\~no}
\address[Monta\~no]{School of Mathematical and Statistical Sciences, Arizona State University, AZ, USA}
\email{montano@asu.edu}

\begin{document}


	\maketitle

	\begin{abstract}
	We describe the twisted $K$-polynomial of multiplicity-free varieties in a multiprojective setting.
	More precisely, for multiplicity-free varieties, we show that the support of the twisted $K$-polynomial is a generalized polymatroid.
	As applications, we show that the support of the M\"obius function of a linear polymatroid is a generalized polymatroid and we settle a conjecture of Monical, Tokcan and Yong regarding Grothendieck polynomials  for the case of zero-one Schubert polynomials.
	\end{abstract}

		\section{Introduction}

		Let $\kk$ be a field, $\PP = \PP_\kk^{m_1} \times_\kk \cdots \times_\kk \PP_\kk^{m_p}$ be a multiprojective space, and $X \subset \PP$ be a subvariety. 
		Following Brion \cite{BRION_MULT_FREE}, we say that $X \subset \PP$ is \emph{multiplicity-free} if the corresponding class $[X]$ in the Chow ring $A^*(\PP) \cong \ZZ[t_1,\ldots,t_p]/\big(t_1^{m_1+1},\ldots,t_p^{m_p+1}\big)$ is a linear combination only admitting  coefficients $0$ or $1$, of the classes of   $\PP_\kk^{a_1}\times_\kk \cdots \times_\kk \PP_\kk^{a_p} \subset \PP$ with $a_1+ \cdots+ a_p = \dim(X)$.
		As proved by Brion \cite{BRION_MULT_FREE}, these varieties have a number of surprising and desirable properties (e.g., they are arithmetically Cohen-Macaulay and normal, and they admit a flat degeneration to a reduced union of products of projective spaces).
		The family of multiplicity-free varieties is large enough that it has become of remarkable importance in applications in algebraic geometry, commutative algebra, representation theory, and combinatorics (see, e.g., \cite{aholt2013hilbert, BRION_FLAG, PERRIN_SPHERICAL, knutson2009frobenius_point, BRION_MULT_FREE, YONG_SCHUBERT_CALC, ESCOBAR_KNUTSON, knutson2009frobenius, PERRIN_LECTURES, li2013images, CS_PAPER,CDNG_CS_IDEALS,CDNG_GIN,CDNG_GRAPH,CDNG_MINORS, caminata2022multidegrees}).
		
		
		The main goal of this paper is to provide a precise description of the class $[X]$ of a multiplicity-free variety $X \subset \PP$ in the Grothendieck ring $K(\PP) \cong \ZZ[t_1,\ldots,t_p]/\big((t_1-1)^{m_1+1},\ldots,(t_p-1)^{m_p+1}\big)$ of coherent sheaves on $\PP$.
		We write the Hilbert series of $X$ as $\Hilb_X(\ttt) = \mathcal{K}(X;\ttt) / (1-t_1)^{m_1+1} \cdots (1-t_p)^{m_p+1}$.
		Following \cite{KNUTSON_MILLER_SCHUBERT}, we say that $\mathcal{K}(X;\ttt)$ is the \emph{$K$-polynomial of $X$} and we call $\mathcal{K}(X;\zzz) = \mathcal{K}(X;1-z_1,\ldots,1-z_p)$ 
		 the \emph{twisted $K$-polynomial of $X$}.
		The coefficients of $\K(X;\zzz)$ correspond to the expression of $[X] \in K(\PP)$ as a linear combination of the classes of $\OO_{\PP_{\kk}^{a_1} \times_\kk \cdots \times_\kk \PP_\kk^{a_p}}$ with $\PP_{\kk}^{a_1} \times_\kk \cdots \times_\kk \PP_\kk^{a_p} \subset \PP$.
		 Our main result is the following theorem.
		
			\begin{headthm}[\autoref{main}]
				\label{thmA}
			Let $X \subset \PP=\PP_\kk^{m_1} \times_{\kk} \cdots \times_{\kk} \PP_\kk^{m_p}$ be a multiplicity-free  variety. 
			Then the support of $\mathcal{K}(X;\zzz)$ is a generalized  polymatroid. 
			Furthermore, we have the specific description that $\bn = (n_1,\ldots,n_p) \in \supp\left(\mathcal{K}(X;\zzz)\right)$ if and only if the following inequality holds
			$$
			\codim_\fJ\left(X, \PP\right) \;\le\; \sum_{j \in \fJ} n_j \;\le\;  b_\fJ(X, \PP)
			$$
			for any nonempty subset $\fJ \subseteq \{1,\ldots,p\}$.
		\end{headthm}

		We now explain the notation in \autoref{thmA}.
		For a given subset $\fJ = \{j_1,\ldots,j_k\} \subseteq [p] = \{1,\ldots,p\}$, we introduce the following notation. 
		We consider the natural projection 
		$$
		\Pi_\fJ :   \PP_\kk^{m_1} \times_\kk \cdots \times_\kk \PP_\kk^{m_p} \rightarrow \PP_\kk^{m_{j_1}} \times_\kk \cdots \times_\kk \PP_\kk^{m_{j_k}}.
		$$
		The codimension of $\Pi_\fJ(X)$ in $\Pi_\fJ(\PP)$ is denoted as 
		$
		\codim_\fJ\left(X, \PP\right) = \codim\left(\Pi_\fJ(X), \Pi_\fJ(\PP)\right). 
		$
		We also define a certain bounding number $b_\fJ(X, \PP)$; see \autoref{nota_bounds} for the precise definition.
		When $\fJ = [p]$, we have the 
		expression $b_{[p]}(X, \PP) = \codim(X, \PP) + \reg\left(\mathcal{S}(X)\right)$, where the second summand denotes the Castelnuovo-Mumford regularity of the graded section ring 
		$$
		\mathcal{S}(X)  \;=\; \bigoplus_{u \in \NN} \left[\mathcal{S}(X)\right]_{u} \quad\text{ where }\quad \left[\mathcal{S}(X)\right]_{u} \;=\; \bigoplus_{(v_1,\ldots,v_p) \in \NN^p, \;\, \sum_{j\in [p]}v_j=u } \HH^0\left(X, \OO_X(v_1,\ldots,v_p)\right).
		$$
		If we write $\K(X;\zzz) = \sum_{\bn \in \NN^p} \, c_\bn(X) \zzz^\bn$,
		then $X \subset \PP$ is multiplicity-free when $c_\bn(X) \in \{0,1\}$ for all $\bn \in \NN^p$ with $|\bn| = \codim(X, \PP)$.
		The support of $\mathcal{K}(X;\zzz)$ is given by $\supp\left(\mathcal{K}(X;\zzz)\right) = \lbrace \bn \in \NN^p \mid c_\bn(X) \neq 0 \rbrace$.
			

		The problem of studying the support of the class of $X\subset \PP $  in $A^*(\PP)$
		 is relatively well-understood. 
		In the projective case (i.e., $p=1$), this problem is vacuous as  there is only one positive coefficient to consider (i.e., the degree of $X \subset \PP_\kk^m$). 
		In the biprojective case (i.e., $p = 2$) we have the works of Trung \cite{trung2001positivity} and Huh \cite{Huh12}. 
		In \cite{POSITIVITY}, we have explicitly described which multidegrees of $X$ are positive in terms of the dimensions of the natural projections of $X$, and we have shown that  they form a polymatroid. 
		Moreover, due to the work of Br\"and\'en and Huh \cite{HUH_BRANDEN}, we have that the volume polynomial of $X$ (a polynomial encoding  the class $[X] \subset A^*(\PP)$)  is Lorentzian.
		
		On the other hand, much less is known regarding the support of the class of $X \subset \PP$ in the Grothendieck ring $K(\PP)$. 
		For the projective case (i.e., $p =1$), in \autoref{subsect_Marley}, we give a simple extension of Marley's work \cite{marley1989coefficients} on the coefficients of the Hilbert-Samuel polynomial of a zero-dimensional ideal in a Noetherian local ring. 
		As a consequence, in \autoref{cor_pos_Marley}, we obtain that the class $[X] \in K(\PP_{\kk}^m)$ of an arithmetically Cohen-Macaulay subvariety $X \subset \PP_\kk^m$ can be written as an alternating sign linear combination of the classes of $\PP_{\kk}^a \subset \PP_{\kk}^m$ and that the corresponding support forms a segment containing all the intermediate points (i.e., it is a generalized polymatroid).
		An extension of Marley's work is not possible in general for the case $p \ge 2$ (see \autoref{rem_Marley_no_multigrad}).
		For the entire multiprojective case (i.e, any $p \ge 1$), under the assumption that $X \subset \PP$ has rational singularities, another work of Brion \cite{BRION_POSITIVE}  provides the natural extension of the aforementioned sign alternation statement (see also the works of  Buch \cite{BUCH} and Anderson, Griffeth and Miller \cite{ANDERSON_GRIFFETH_MILLER}).
		In \autoref{cor_positivity_shellable}, we settle this sign alternation result for varieties with a square-free Gr\"obner degeneration whose corresponding simplicial complex is shellable.
		We point out that multiplicity-free varieties have rational singularities and their multigraded generic initial ideals yield shellable simplicial complexes.
		To the best of  our knowledge,  \autoref{thmA} is the first result that addresses the following question: 
	\begin{changemargin}{.7cm}{.7cm} 
		\emph{Given a subvariety $X \subset \PP$, when do we have that $\supp\left(\K(X;\zzz)\right)$ is a generalized polymatroid?}
	\end{changemargin}

		
		In \autoref{sect_results_mult_free}, we present additional results for multiplicity-free varieties. 
		There, we slightly change the notation as it is most convenient to work in multiprojective geometric terms.
		For each $\bn = (n_1,\ldots,n_p) \in \NN^p$, we denote by $\deg_\PP^\bn(X)$ the corresponding \emph{Hilbert coefficient} (see \autoref{sect_recap_multdeg}), when $n_1+\cdots+n_p = \dim(X)$ we have that $\deg_\PP^\bn(X)$ coincides with the usual notion of \emph{multidegree}.
		For a multiplicity-free variety $X \subset \PP$, 
		we prove the following results:
		\begin{enumerate}[\rm (i)]
			\item In \autoref{lem_eq_Hilb_pol_funct_mult_free}, we show that the Hilbert polynomial equals the Hilbert function everywhere in $\NN^p$. 
			\item In \autoref{lemma_Lex}, we prove that any lexicographical ordering of the points in $\msupp_\PP(X) = \lbrace \bn \in \NN^p \mid |\bn| = \dim(X) \text{ and }  \deg_\PP^\bn(X) \neq 0\rbrace$ induces a shelling on the simplicial complex corresponding to the multigraded generic initial ideal (which is square-free for multiplicity-free varieties, see \autoref{thm_gin_mult_free}).
			This is an important technical result for us. 
			\item In \autoref{lem_stalactite_mult_free}, we show 
			that the whole Hilbert polynomial of $X$ is determined by $\msupp_\PP(X)$, and we provide an algorithmic process to compute it. 
			\item \autoref{prop_mult_free_statements} contains several technical results.
			For instance,  part (iv) gives an alternative proof of a result of Knutson \cite{knutson2009frobenius} yielding a M\"obius-like formula: for any $\bn \in \NN^p$ such that there exists $\bw \in \msupp_\PP(X)$ with $\bw \ge \bn$, we obtain the equality 
			$
			\sum_{\bw \ge \bn} \deg_\PP^\bw(X) \;=\; 1.
			$
			\item In \autoref{thm_hsupp_polymatroid}, we show that $\hsupp_\PP(X) = \lbrace \bn \in \NN^p \mid \deg_\PP^\bn(X) \neq 0 \rbrace$ is a generalized polymatroid.
		\end{enumerate}
		Our results in \autoref{sect_results_mult_free} rely on the ones obtained in \autoref{sect_shellable} and \autoref{sect_combinatorial}.
		\autoref{sect_shellable} contains several results regarding the Hilbert polynomial of a square-free monomial ideal such that the corresponding simplicial complex is shellable (see \autoref{shell_prop}, \autoref{cor_positivity_shellable}, \autoref{cor_Mobius_shellable_ideals} and \autoref{cor_path_connected}).
		\autoref{sect_combinatorial} provides a gluing result 
		for generalized polymatroids that allows us to conclude that $\hsupp_\PP(X)$ is a generalized polymatroid (see \autoref{thm_glue_poly}); this is the most involved 
		combinatorial part of our work and for better exposition we introduce certain combinatorial terms: the {\it stalactites} and  the {\it caves}.

		\medskip

		We now present some applications of our main result: \autoref{thmA}.
		
		\medskip
		
		\emph{Schubert} and {\it Grothendieck polynomials} were introduced by Lascoux and Sch\"utzenberger \cite{SCUBERT_POLY_L_S, LASCOUX_SCHUTZENGERGER} (for precise definitions, see \autoref{subsect_Groth_poly}).
		These polynomials have played a central role in algebraic combinatorics (see,~e.g.,~\cite{fink2018schubert, monical2019newton,BERGERON_BILLEY,BILLEY_STANLEY, KNUTSON_MILLER_SUBWORD, KNUTSON_MILLER_SCHUBERT}).
		In \cite{monical2019newton}, Monical, Tokcan and Yong made conjectures on the Newton polytopes of several combinatorially defined polynomials.
		In particular, they conjectured that Schubert and Grothendieck polynomials have the saturated Newton polytope property (see \cite[\S 5]{monical2019newton}).
		These conjectures have been settled in the positive for Schubert polynomials: 
		in \cite{fink2018schubert}, it was shown that the support of an ordinary Schubert polynomial is a polymatroid;
		in \cite{DOUBLE_SCHUBERT}, it was shown that the support of a double Schubert polynomial is also a polymatroid.
		More recently,  Huh--Matherne--M\'esz\'aros--St.~Dizier \cite{HuhLogarithmic} and M\'esz\'aros--Setiabrata--St.~Dizier \cite{meszaros2022support} conjectured that the support of a Grothendieck polynomial is a generalized polymatroid, and they proved it for a special class of permutations called Grassmannian.
		Here we settle another important case.

\begin{headthm}[\autoref{thm:grothendieck}]
	\label{thmB}
	Let $w \in S_p$  be a permutation such that $\mathfrak{S}_w$ is a zero-one Schubert polynomial, then the support $\supp(\mathfrak{G}_{w})$ of the Grothendieck polynomial $\mathfrak{G}_{w}$ is a generalized polymatroid. 
\end{headthm}

	Our proof of \autoref{thmB} follows by combining \autoref{thmA} and the work of Knutson and Miller \cite{KNUTSON_MILLER_SCHUBERT}.
	By \cite[Theorem A]{KNUTSON_MILLER_SCHUBERT}, the Schubert polynomial $\mathfrak{S}_w$ equals the multidegree polynomial $\mathcal{C}(\overline{X}_w;\zzz)$ and the Grothendieck polynomial $\mathfrak{G}_w$ coincides with the twisted $K$-polynomial $\mathcal{K}(\overline{X}_w;\zzz)$. 
	Here $\overline{X}_w$ denotes the \emph{matrix Schubert variety} of $w$ which  can be seen as a subvariety of the product of projective spaces $\big(\PP_{\kk}^{p-1}\big)^p$.
	If $\mathfrak{S}_w$ is a zero-one Schubert polynomial (see \cite{FINK_MESZAROS_DIZIER}), then $\overline{X}_w$ is a multiplicity-free variety.
	Notice that \autoref{thmA} also provides defining inequalities for the Newton polytope of the Grothendieck polynomial $\mathfrak{G}_w$.
	The description of these inequalities is in terms of algebraic invariants of the corresponding matrix Schubert variety $\overline{X}_w$.
	It would be interesting to find  combinatorial analogues of these inequalities.
	In particular, the global upper bound (i.e., when $\fJ = [p]$) involves the regularity of the matrix Schubert variety $\overline{X}_w$ which has been recently studied in \cite{pechenik2021castelnuovo} from a combinatorial perspective.
	
	We point out that our general approach under shellability assumptions (as in \autoref{sect_shellable}) could be used to study arbitrary Grothendieck polynomials.
	For instance, since the initial ideal of the \emph{Schubert determinantal ideal} $I_w$ with respect to the antidiagonal order is square-free and the corresponding simplicial complex is shellable (see \cite[\S 16.5]{miller2005combinatorial}), \autoref{cor_positivity_shellable} gives yet another proof of the alternation of sign of the coefficients in Grothendieck polynomials (also, see \autoref{cor_Mobius_shellable_ideals} and \autoref{cor_path_connected} for further possible consequences).
	We restrict ourselves to the case of multiplicity-free varieties because the corresponding multigraded generic initial ideals are square-free and their associated primes are in bijection with the positive multidegrees (see \autoref{rem_corr_prime_mult}).
	This latter fact is a stepping stone for our combinatorial considerations in \autoref{sect_combinatorial}.


	\emph{Polymatroids} are fundamental objects in the areas of combinatorics and optimization (see \cite[Chapter 12]{HERZOG_HIBI} and \cite[Part IV]{schrijver2003combinatorial}).
	They are a 
	 natural extension of \emph{matroids}. 
	Polymatroids  
	 have also been studied  under the names of \emph{M-convex sets} \cite{MUROTA} and  \emph{generalized permutohedron} \cite{POSTNIKOV}.
	There are 
	two distinguished classes of  polymatroids, \emph{linear} and {\it algebraic polymatroids}, whose main properties are inherited by  their respective representations. 
	By utilizing \autoref{thmA}, we obtain interesting consequences for linear polymatroids. 
	The general idea is that we can use multiplicity-free varieties as a model to study linear polymatroids.
	Indeed, from the work of Li \cite{li2013images}, for a given linear polymatroid $\mathscr{P}$ on $[p]$ we can find a \emph{multiview variety}  $X \subset \PP = \PP_{\kk}^{m_1} \times_\kk \cdots \times_\kk \PP_\kk^{m_p}$ (a special type of multiplicity-free variety) that satisfies $\mathscr{P} = \msupp_\PP(X)$ (see \autoref{prop_linear_polymatroid}).
	When $\mathscr{P}$ is a matroid on $[p]$, we can assume that $X$ is  embedded in the product $\big(\PP_\kk^1\big)^p$ of projective lines (see \cite{ARDILA_BOOCHER}). 
	Our main result in this direction is the following theorem regarding the M\"obius function of the partially ordered set associated to a polymatroid.

\begin{headthm}[\autoref{thm:linear_polymatroid}]
	\label{thmC}
Let  $ \mathscr{P} $ be a linear polymatroid.
Consider  the set 
$$
\mathscr{P}_{\le} \;=\; \{ \fu\in\NN^p ~\mid~ \fu \ls \fv \text{ for some }   \fv\in \mathscr{P}\}
$$ 
and the partially ordered set $\Gamma$ on it given by componentwise ordering, with a maximum $\hat{1}$ added.
Let $\mu_\Gamma$ be the M\"obius function of $\Gamma$.
Then the set
\[
	\mu\text{\rm-supp}(\mathscr{P}) \;=\; \left\{\fu \in \mathscr{P}_\leq ~\mid~ \mu_\Gamma(\fu,\hat{1}) \neq 0 \right\}
\]
is a generalized polymatroid.
\end{headthm}
We expect the result of \autoref{thmC} to hold for arbitrary polymatroids (see \autoref{conjecture}).

	\medskip
	
	\noindent
	\textbf{Outline.} 
	The structure of the paper is as follows.
	In \autoref{sect_prelim}, we fix the notation and recall some results that are needed throughout the rest of the paper.
	\autoref{sect_motivation} presents two results that initially motivated our approach: \autoref{subsect_Marley} contains an extension of Marley's work \cite{marley1989coefficients}, and \autoref{subsect_Brion} provides a short account of Brion's positivity result \cite{BRION_POSITIVE} restricted to a multiprojective setting.
	In \autoref{sect_shellable}, we study square-free monomial ideals whose corresponding simplicial complexes are shellable.
	We provide our combinatorial approach for gluing generalized polymatroids in \autoref{sect_combinatorial}.
	Our technical results regarding multiplicity-free varieties are proved in \autoref{sect_results_mult_free}.
	Finally, \autoref{sect_proof_main} contains the proofs of our main results: \autoref{thmA}, \autoref{thmB}, and \autoref{thmC}.

	\section{Preliminaries}
	\label{sect_prelim}
	
	\noindent\textbf{Notation.} We first fix the notation that is used throughout the article.  
	We let $[p]:=\{1,\ldots,p\}$ for  $p\in \NN$ and we denote the standard basis vectors of $\mathbb{R}^n$ by $\be_1,\ldots,\be_p$. 
	Given any vector $\ba\in \mathbb{R}^p$, we denote $[\ba]_i$ for its $i^{\rm th}$ coordinate and set $\supp(\ba):=\{i\mid [\ba]_i\neq 0\}$. 
	We also denote $|\ba|$ for the $1$-norm of $\ba$, i.e.,~$\sum_{i=1}^p [\ba]_i$. 
	For a pair of vectors $\ba$ and $\bb$ we write $\ba\leq \bb$ if $[\ba]_i\leq [\bb]_i$ for all $i$.  
	We denote by $\mathbf{0}$ the vector $(0,\ldots, 0)\in \NN^p$.

	Following \cite{monical2019newton}, we say that a polynomial $f = \sum_{\bn}c_{\bn}\zzz^\bn\in \ZZ[z_1,\ldots,z_p]$ has the \emph{Saturated Newton Polytope property} (SNP property for short) if the support  $\supp(f)=\{\bn\in\NN^p\mid c_\bn \neq 0\}$ of $f$ is equal to $\text{Newt}(f)\cap\NN^p$, where $\text{Newt}(f) := \text{ConvexHull}\{\bn\in\NN^p\mid c_\bn \neq 0\}$ denotes the \emph{Newton polytope} of $f$; in other words, if the support of $f$ consists of the integer points of a polytope.

	\subsection{A short recap on multidegrees}
		\label{sect_recap_multdeg}
		
		Here we briefly recall the notion of multidegrees and some of its basic properties (for more details the reader is referred to \cite{miller2005combinatorial, cidruiz2021mixed}).
		
		Let $p >0$ be a positive integer, $\kk$ be a field and $S = \kk[x_1,\ldots,x_n]$ be a standard graded $\NN^p$-graded polynomial ring (i.e., $\deg(x_i) \in \NN^p$ and $|\deg(x_i)| = 1$ for all $i \in [p]$).  
		Let $M$ be a finitely generated $\ZZ^p$-graded $S$-module and $F_\bullet$ be a $\ZZ^p$-graded free $S$-resolution 
		$
		F_\bullet : \; \cdots \rightarrow F_i \rightarrow F_{i-1} \rightarrow \cdots \rightarrow F_1 \rightarrow F_0
		$
		of $M$.
		Let $t_1,\ldots,t_p$ be variables over $\ZZ$ and consider the polynomial ring $\ZZ[\ttt] = \ZZ[t_1,\ldots,t_p]$, where the variable $t_i$ corresponds with the $i$-th elementary vector $\ee_i \in \ZZ^p$. 
		If we write $F_i = \bigoplus_{j} S(-\mathbf{b}_{i,j})$ with $\mathbf{b}_{i,j} = (\mathbf{b}_{i,j,1},\ldots,\mathbf{b}_{i,j,p}) \in \ZZ^p$, then we define the Laurent polynomial 
		$
		\left[F_i\right]_\ttt \, := \, \sum_{j} \ttt^{\mathbf{b}_{i,j}} = \sum_{j} t_1^{\mathbf{b}_{i,j,1}} \cdots t_p^{\mathbf{b}_{i,j,p}}.
		$
		The \emph{Hilbert function} of $M$ is denoted by $h_M : \ZZ^p \rightarrow \NN$ with  $h_M(\fv) := \dim_\kk\left(\left[M\right]_\fv\right)$ for any $\fv \in \ZZ^p$.
		
		\begin{definition}
			Under the above notation, we have the following notions: 
			\begin{enumerate}[\rm (i)]
				\item The \emph{K-polynomial} of $M$ is given by 
				$
				\mathcal{K}(M;\ttt) \, := \, \sum_{i} {(-1)}^i \left[ F_i \right]_\ttt.
				$
				\item The {\it twisted $K$-polynomial} of $M$ is defined as $\K(M;\zzz) = \K(M;1-z_1,\ldots,1-z_p)$, and we write $\K(M;\zzz) = \sum_\bn c_\bn(M) \,\zzz^\bn$.
				\item The \emph{multidegree polynomial} of $M$ is the homogeneous polynomial $\mathcal{C}(M; \zzz) \in \ZZ[\zzz]$ given as the sum of all terms in 
				$
				\mathcal{K}(M; \zzz)
				$
				of smallest degree, which is equal to $\codim(M) = n - \dim(M)$.
			\end{enumerate}
		\end{definition}
		
		
		Another approach to study multidegrees is via the use of multigraded Hilbert polynomials.
		For each $i \in [p]$, suppose that $S$ has $m_i+1$ variables with degree equal to $\ee_i$.
		Thus we have $\multProj(S) = \PP := \PP_\kk^{m_1} \times_\kk \cdots \times_\kk \PP_\kk^{m_p}$.
		We denote the (multigraded) \emph{Hilbert polynomial} of $M$ by $P_M(\ttt)=P_M(t_1,\ldots,t_p) \in \QQ[\ttt]=\QQ[t_1,\ldots,t_p]$.
		Then, the degree of $P_M$  equals the dimension $\dim(\widetilde{M})$ of the correspondent coherent sheaf $\widetilde{M}$ on $\PP$ and it satisfies 
		$
		P_M(\fv) = h_M(\fv) 
		$
		for all $\fv \in \NN^p$ such that $\fv \gg \mathbf{0}$.
		Furthermore, if we write 
		\begin{equation*}
			\label{eq_Hilb_poly}
			P_{M}(\ttt) \;=\; \sum_{\bn\in \NN^p} e_\bn(M)\binom{t_1+n_1}{n_1}\cdots \binom{t_p+n_p}{n_p},
		\end{equation*}
		then $e_\bn(M)\in \ZZ$ for all $\bn$ and $e_\bn(M) \ge 0$ when $|\bn| = \dim\big(\widetilde{M}\big)$.	
		
		\begin{definition}
			For any $\bn \in \NN^p$, we say that $e_\bn(M)$ is the \emph{Hilbert coefficient of $M$ of type $\bn$}.
			In particular, when $|\bn| = \dim(\widetilde{M})$, we say that $e_\bn(M)$ is the \emph{mixed multiplicity of $M$ of type $\bn$}.
		\end{definition}
		
		Let $J \subset S$ be an $S$-homogeneous ideal that is saturated with respect to the multigraded irrelevant ideal of $S$, and consider the standard $\NN^p$-graded $\kk$-algebra $R = S/J$.
		Let 
		$$
		X = \multProj(R)\subset \PP = \PP_\kk^{m_1} \times_\kk \cdots \times_\kk \PP_\kk^{m_p} = \multProj(S)
		$$ 
		be the corresponding closed subscheme.
		The \emph{$K$-polynomial} and the \emph{multidegree polynomial} of $X$ are given by $\mathcal{K}(X;\ttt) := \mathcal{K}(R;\ttt)$ and  $\mathcal{C}(X;\ttt) := \mathcal{C}(R;\ttt)$, respectively.
		Similarly, the \emph{Hilbert polynomial} of $X$ is given by $P_X(\ttt) := P_R(\ttt)$.
		
		\begin{definition}
			Under the above notation, we have the following notions: 
			\begin{enumerate}[\rm (i)]
				\item For any $\bn \in \NN^p$, we say that $\deg_\PP^\bn(X):=e_\bn(R)$ is the \emph{Hilbert coefficient of $X$ of type $\bn$}.
				In particular, when $|\bn| = \dim(X)$, we say that $\deg_\PP^\bn(X)$ is the \emph{multidegree of $X$ of type $\bn$}.
				\item We say that the \emph{Hilbert support of $X$} is given by the set 
				$$
				\hsupp_\PP(X) \;:=\; \big\{\bn \in \NN^p  \mid \deg_\PP^\bn(X) \neq 0 \big\}.
				$$
				\item We say that the \emph{multidegree support of $X$} is given by the set 
				$$
				\msupp_\PP(X) \;:=\; \big\{\bn \in \NN^p  \mid |\bn| = \dim(X) \text{  and } \deg_\PP^\bn(X) \neq 0 \big\}.
				$$
			\end{enumerate}
		\end{definition}

		Under the above notation, we have the equality 
		$$
		\mathcal{C}(X;\zzz) \; = \; \sum_{\substack{\bn \in \NN^p\\ |\bn| = \dim(X)}} \deg_\PP^\bn(X) \, z_1^{m_1-n_1} \cdots z_p^{m_p-n_p};
		$$
		see \cite[Remark 2.9]{POSITIVITY}.

		In \autoref{lem_K_poly_reduce_Hilb}, we shall prove further connections between Hilbert coefficients and the $K$-polynomial.

	Following standard notations, we say that $X$ is a \emph{variety} over $\kk$ if $X$ is a reduced and irreducible separated scheme of finite type over $\kk$ (see, e.g., \cite[\href{https://stacks.math.columbia.edu/tag/020C}{Tag 020C}]{stacks-project}). 
	Since multidegrees are non-negative integers, it becomes natural to address the positivity of these numbers. 
	We now recall a complete characterization for the positivity of multidegrees from \cite{POSITIVITY}, which is an important tool in our work.
	For each subset $\mathfrak{J} = \{j_1,\ldots,j_k\} \subseteq [p]$, consider the natural projection $\Pi_\fJ :  \PP_\kk^{m_1} \times_\kk \cdots \times_\kk \PP_\kk^{m_p} \rightarrow \PP_\kk^{m_{j_1}} \times_\kk \cdots \times_\kk \PP_\kk^{m_{j_k}}$ and denote by $R_{(\fJ)}$ the $\ZZ^k$-graded $\kk$-algebra given by 
	$$
	R_{(\fJ)} := \bigoplus_{\substack{i_1\ge 0,\ldots, i_p\ge 0\\ i_{j} = 0 \text{ if } j \not\in \fJ}} {\left[R\right]}_{(i_1,\ldots,i_p)}.
	$$
	The following theorem characterizes the positivity of multidegrees.

	\begin{theorem}[{\cite[Theorem A, Proposition 5.1]{POSITIVITY}}]
		\label{thm_pos_multdeg}
		Let $X \subset \PP = \PP_\kk^{m_1} \times_\kk \cdots \times_\kk \PP_\kk^{m_p}$ be a variety.
		Then, for any $\bn = (n_1,\ldots,n_p) \in \NN^p$ with $|\bn| = \dim(X)$, we have that $\deg_\PP^\bn(X) > 0$ if and only if 
		$
		\sum_{j \in \fJ}n_j \;\le \; \dim\left(\Pi_\fJ(X)\right)
		$ 
		for all $\fJ \subseteq [p]$.
		Furthermore, $\msupp_\PP(X)$ is an algebraic 
		polymatroid.
	\end{theorem}

	\subsection{Running example}\label{sec_running_example}
	The following is our  running example.
	Consider the permutation $ w = [1,5,3,2,4]$ in $S_5 $.
	Its {\it Rothe diagram} is depicted in  \autoref{rothe}.

	\begin{figure}[ht]
		\includegraphics[scale=0.4]{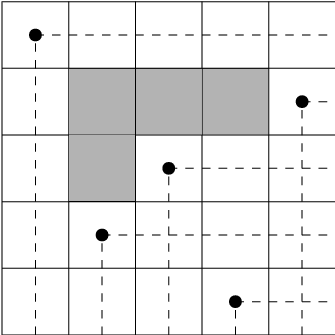}
		\caption{Rothe diagram of $ w = [1,5,3,2,4] $.}
		\label{rothe}
	\end{figure}
	
	The corresponding  \emph{Schubert determinantal ideal} is given by
	\begin{equation*}\label{eq:ideal-example}
		I_w = I_2
		 \begin{pmatrix}x_{1,1}&x_{1,2}&x_{1,3}&x_{1,4} \\  x_{2,1}&x_{2,2}&x_{2,3}&x_{2,4} \end{pmatrix}  
		+
	I_2\begin{pmatrix}x_{1,1}&x_{2,1}&x_{3,1} \\  x_{1,2}&x_{2,2}&x_{3,2}
		\end{pmatrix}\subset S,
	\end{equation*}
	where $I_2(-)$ denotes the ideal generated by $2\times 2$ minors.	
	Each multidegree of the corresponding matrix Schubert variety $X=\multProj(S/I_w) \subset \big(\PP_{\kk}^4\big)^5$ is either zero or one, i.e., $X$ is a multiplicity-free variety (see \autoref{def_multideg}).
	We have that $$
	\msupp_\PP(X) = \big\lbrace
	(1,3,4,4,4),(1,4,3,4,4),(2,2,4,4,4),(2,3,3,4,4),(3,1,4,4,4),(3,2,3,4,4),(4,1,3,4,4) \big\rbrace.
	$$

	\section{Motivating results in the literature}
	\label{sect_motivation}
	
	\subsection{Marley's result for Cohen-Macaulay modules in a single graded setting}
	\label{subsect_Marley}
	
	In this subsection, we extend Marley's work \cite{marley1989coefficients} on the coefficients of the Hilbert-Samuel polynomial of a zero-dimensional ideal in a Noetherian local ring. 
	Here we prove the equivalent of Marley's result for any module over a standard graded algebra. 
	
	\begin{setup}
		\label{setup_Marley}
		Let $(A, \nnn)$ be an Artinian local ring, $R$ be a standard $\NN$-graded finitely generated $A$-algebra, and $\mm = [R]_+ = \bigoplus_{n \ge 1} [R]_n$ be the graded irrelevant ideal of $R$.
	\end{setup}

	Let $M$ be a finitely generated $\ZZ$-graded $R$-module.
	We denote the \emph{Hilbert function of $M$} as 
	$$
	h_M : \ZZ \rightarrow \ZZ, \quad h_M(n) := \text{length}_A\left([M]_n\right). 
	$$
	Let $d = \dim(M)$ be the dimension of $M$.
	The \emph{Hilbert polynomial of $M$} can be written as 
	$$
	P_M(t) \;=\; \sum_{i=0}^{d-1}e_i(M) \binom{t+i}{i} \;\in\; \QQ[t], 
	$$
	and it satisfies $e_{i}(M) \in \ZZ$ for all $i \ge 0$, $e_{d-1}(M) >0$, and   $P_M(n) = \text{length}_A\left([M]_n\right)$ for $n \gg 0$ (see, e.g., \cite[\S 4.1]{BRUNS_HERZOG}).
	The integers $e_0(M), \ldots, e_{d-1}(M)$ are called the \emph{Hilbert coefficients of $M$}.
	In particular, $e(M):=e_{d-1}(M)$ is called the \emph{multiplicity of $M$}.

	Given a function $f: \ZZ \rightarrow \ZZ$, we define its first difference function as 
	$$
	\Delta^1(f) : \ZZ \rightarrow \ZZ,  \quad \Delta^1(f)(n) := f(n+1) - f(n).
	$$
	Inductively, we set $\Delta^i(f) := \Delta^1\left(\Delta^{i-1}(f)\right)$. 
	Also, we take the convention $\Delta^0(f) := f$.
	The main result of this subsection is the following extension of \cite[Theorem 1]{marley1989coefficients}.

	\begin{theorem}
		\label{thm_Marley}
		Assume \autoref{setup_Marley}.
		Let $M$ be a finitely generated $\ZZ$-graded $R$-module of dimension $d = \dim(M)$.
		Suppose that $M$ is a Cohen-Macaulay module. 
		Then the following inequality 
		$$
		{(-1)}^{d-i}\Delta^{i}\left(h_M-P_M\right)(n) \;\ge\; 0
		$$
		holds for all $0 \le i \le d$ and $n \in \ZZ$.
	\end{theorem}
	\begin{proof}
		By utilizing the (standard) faithfully flat extension $A \rightarrow A[x]_{\nnn A[x]}$, we may assume that the residue field of $A$ is infinite.
		Set $f:=h_M - P_M$.
		
		First, we prove that $\Delta^{d}(f)(n) \ge 0$.
		To that end, we proceed by induction on $d = \dim(M)$.
		If $d=0$, then $P_M(t) = 0$, and so the inequality $\Delta^0(f)(n) = h_M(n) - P_M(n) \ge 0$ holds for all $n \in \ZZ$.
		Suppose that $d \ge 1$.
		Since the residue field of $A$ is infinite and $M$ is Cohen-Macaulay, we may find a general element $\ell
		\in [R]_1$ which is a nonzerodivisor over $M$.
		It follows that $M/\ell M$ is a Cohen-Macaulay module of dimension $d-1$ and that the equalities $\Delta^1(h_M)(n) = h_{M/\ell M}(n+1)$ and $\Delta^1(P_M)(n) = P_{M/\ell M}(n+1)$ hold.
		Therefore, the induction hypothesis yields that
		$$
		\Delta^d(f)(n) = \Delta^{d-1}\left(\Delta^1(f)\right)(n) = \Delta^{d-1}(h_{M/\ell M} - P_{M/\ell M})(n+1) \ge 0
		$$
		for all $n \in \ZZ$.
		Hence we have settled the case $i = d$.
		
		Let $0 < i \le d$.
		Suppose that we have showed ${(-1)}^{d-i}\Delta^i(f)(n) \ge 0$ for all $n \in \ZZ$.
		On the other hand, we know that $f(n) = 0$ for all $n \gg 0$.
		Thus, we have
		$$
		{(-1)}^{d-i}\Delta^{i-1}(f)(n+1) \,-\, {(-1)}^{d-i}\Delta^{i-1}(f)(n) \,\ge\, 0 \quad \text{for all $n \in \ZZ$}
		$$
		and  $\Delta^{i-1}(f)(n) = 0$ for all $n \gg 0$.
		This implies that ${(-1)}^{d-(i-1)}\Delta^{i-1}(f)(n) \ge 0$ for all $n \in \ZZ$, finishing the proof. 
	\end{proof}	
	
	\begin{corollary}
		\label{cor_stabilization_Marley}
		Let $M$ be a Cohen-Macaulay finitely generated $\ZZ$-graded $R$-module of positive dimension.
		If $h_M(k) = P_M(k)$ for some $k \in \ZZ$, then $h_M(n) = P_M(n)$ for all $n \ge k$.
	\end{corollary}
	\begin{proof}
		Let $d = \dim(M) \ge 1$.
		By \autoref{thm_Marley}, we have ${(-1)}^{d-1}\Delta^1(h_M-P_M)(n)\ge 0$ for all $n \in \ZZ$, and thus we obtain 
		$$
		0 \;\ge\; {(-1)}^{d-1}(h_M - P_M)(n)  \;\; \text{ for all $n \in \ZZ$}
		$$
		due to the fact that $h_M(n) = P_M(n)$ for all $n \gg 0$.
		Therefore 
		$$
		0 \;\ge\; {(-1)}^{d-1}(h_M - P_M)(n) \;\ge\;   {(-1)}^{d-1}(h_M - P_M)(k) \;=\; 0
		$$
		for all $n \ge k$, and the result follows.
	\end{proof}
	
	Finally, the following corollary provides positivity and rigidity results for the Hilbert coefficients of a Cohen-Macaulay module.
	
	\begin{corollary}
		\label{cor_pos_Marley}
		Let $M$ be a Cohen-Macaulay finitely generated $\ZZ$-graded $R$-module of dimension $d = \dim(M)$.
		Then the following statements hold:
		\begin{enumerate}[\rm (i)]
			\item ${(-1)}^{d-1-j}e_j(M) \ge 0$ for all $0 \le j \le d-1$.
			\item For $0 \le j <  i < d-1$, if $e_i(M) = 0$ then $e_{j}(M) = 0$.
		\end{enumerate}
	\end{corollary}
	\begin{proof}
		As in \autoref{thm_Marley}, we may assume that the residue field of $A$ is infinite, and thus may find a  regular sequence $\ell_1,\ldots,\ell_j \in [R]_1$ over $M$.
		Then $M'=M/(\ell_1,\ldots,\ell_j)M$ is Cohen-Macaulay of dimension $d-j$ and we have the equality $e_k(M) = e_{k-j}(M')$ for all $k \ge j$.
		Hence we may assume that $j = 0$ to prove both parts of the corollary.
		
		(i) From \autoref{thm_Marley}, we get the inequality
		$$
		0 \;\le\; {(-1)}^d\left(h_M(-1) - P_M(-1)\right) \;=\; {(-1)}^d\left(0-e_0(M)\right) \;=\; {(-1)}^{d-1}e_0(M)
		$$
		that completes the proof of this part.
		
		(ii) It suffices to show that: if $e_1(M)=0$ then $e_0(M)=0$.
		Thus assume that $e_1(M)=0$.
		Choose a nonzerodivisor $\ell \in [R]_1$ over $M$.
		Then $e_0(M/\ell M) = e_1(M) = 0$, $\Delta^1(h_M)(n) = h_{M/\ell M}(n+1)$, and $\Delta^1(P_M)(n) = P_{M/\ell M}(n+1)$.
		Consequently, $P_{M/\ell M}(-1) = e_0(M/\ell M) = 0$, and \autoref{cor_stabilization_Marley} implies  that $h_{M/\ell M}(n) = P_{M/\ell M}(n)$ for all $n \ge -1$.
		Inductively, since $h_{M}(n) = P_{M}(n)$ for $n\gg 0$, we obtain  
		$$
		h_{M}(n) \;=\; P_{M}(n)  \;\; \text{for all $n \ge -2$}.
		$$
		Therefore, $P_M(-1)=0$ and this implies that $e_0(M)=0$, finishing  the proof.
	\end{proof}

	\begin{example}
		\label{rem_Marley_no_multigrad}
		Here we discuss a basic reason why an extension of Marley's work \cite{marley1989coefficients} is not possible for the case $p \ge 2$.
		Let $S = \kk[x_0,\ldots,x_r]$ and $I = (f_0,\ldots,f_s) \subset S$ be an ideal generated by homogeneous polynomials of the same positive degree.
		The Rees algebra $\Rees(I) = \bigoplus_{n=0}^\infty I^nt^n$ can be written as a quotient of the standard bigraded polynomial ring $T = \kk[x_0,\ldots,x_r,y_0,\ldots,y_s]$.
		It gives the bihomogeneous coordinate ring of the graph $\Gamma \subset \PP_{\kk}^r \times_\kk \PP_{\kk}^s$ of the rational map $\mathcal{F} : \PP_{\kk}^r \dashrightarrow \PP_\kk^s$ with representative $(f_0:\cdots:f_s)$.
		Let $\mm_1 = (x_0,\ldots,x_r) \Rees(I)$, $\mm_2 = (y_0,\ldots,y_s) \Rees(I)$, and $Y \subset \PP_{\kk}^s$ be the image of $\mathcal{F}$.
		To emulate the approach in the proof of \autoref{thm_Marley} for the bigraded algebra $\Rees(I)$, we would need to successively cut by nonzerodivisors that are either in $\mm_1$ or $\mm_2$.
		However, we have very few such elements in this case.
		Indeed, we have that $\codim(\mm_2) = 1$  and  $\codim(\mm_1) = (r+1) - \dim(Y)$ (see, e.g., \cite[Chapter 5]{SwHu}); in particular, we also obtain that $\codim(\mm_1) = 1$ when $\mathcal{F}$ is generically finite.
 	\end{example}
	
	\subsection{Brion's positivity result for varieties with rational singularities}
	\label{subsect_Brion}
	
	In this short subsection, we quickly present Brion's positivity result \cite{BRION_POSITIVE} for varieties with rational singularities (also, see \cite{BRION_FLAG}), and we restrict ourselves to a multiprojective setting. 
	Since this result has motivated our positivity result for ideals with a shellable initial ideal, we provide a short account.

	We recall the definition of rational singularities (for more details, see \cite[\S 5.1]{KOLLAR_MORI}).
	Let $X$ be a variety over an algebraically closed field $\kk$ of characteristic zero, and $f : \widetilde{X} \rightarrow X$ be a resolution of singularities. 
	We say that $f : \widetilde{X} \rightarrow X$ is a \emph{rational resolution} if 
	\begin{enumerate}[(a)]
		\item $f_* \OO_{\widetilde{X}} \cong \OO_X$ (equivalently, $X$ is normal), and
		\item $R^if_*\OO_{\widetilde{X}} =0$ for $i >0$.
	\end{enumerate}
	We say that $X$ has \emph{rational singularities}  if every resolution $f : \widetilde{X} \rightarrow X$ is rational.
	
	\medskip
	We present the following result of Brion that partly motivated our positivity results. 
	
	\begin{theorem}[{Brion \cite{BRION_POSITIVE}}]
		\label{thm_Brion}
		Let $X \subset \PP_\kk^{m_1} \times_{\kk} \cdots \times_{\kk} \PP_\kk^{m_p}$ be  variety with rational singularities. 
		Then the Hilbert coefficients of $X$ alternate in sign.
		More precisely, we have the inequality
		$$
		(-1)^{\dim(X) - |\bn|} \, \deg_\PP^\bn(X) \,\ge\, 0
		$$
		for all $\bn \in \NN^p$.
	\end{theorem}
	\begin{proof}
		Let $f : \widetilde{X} \rightarrow X$ be a rational resolution of $X$.
		Consider the ample line bundle $\mathcal{L} = \OO_X(1, \ldots,1)$.
		From the rational singularities assumption and the projection formula (see, e.g., \cite[Exercise III. 8.3]{HARTSHORNE}), it follows that the Leray spectral sequence
		$$
		E_2^{p,q} \;=\; \HH^p\big(X,\, R^qf_*\left(f^*\mathcal{L}^{-1}\right)\big) \;\cong\; \HH^p\big(X,\, R^qf_*\OO_{\widetilde{X}} \otimes_{\OO_X} \mathcal{L}^{-1}\big) \; \Longrightarrow \; \HH^{p+q}\big(\widetilde{X}, f^*\mathcal{L}^{-1}\big)
		$$
		degenerates, and so we get the isomorphism $\HH^i(X, \mathcal{L}^{-1}) \cong \HH^i(\widetilde{X}, f^*\mathcal{L}^{-1})$ for all $i \ge 0$.
		The Kawamata--Viehweg vanishing theorem (see, e.g., \cite[\S 5]{VANISHING_THMS}) implies that $\HH^i(X, \mathcal{L}^{-1}) \cong \HH^i(\widetilde{X}, f^*\mathcal{L}^{-1}) = 0$ for all $i < r:=\dim(X)$.
		By utilizing this vanishing and the expression of the (multigraded) Hilbert polynomial in terms of Euler characteristic (see, e.g., \cite[Lemma 4.3]{KLEIMAN_THORUP})
		$$
		P_X(n_1,\ldots,n_p) \;=\; \sum_{i=0}^{r} (-1)^{i} \dim_\kk\left(\HH^i\big(X, \OO_X(n_1,\ldots,n_p)\big)\right),
		$$
		we obtain the equalities 
		$$
		\deg_\PP^{\mathbf{0}}(X) \;=\; P_X(-1,\ldots,-1) = (-1)^{r} \dim_\kk\left(\HH^{r}\big(X, \OO_X(-1,\ldots,-1)\big)\right). 
		$$
		Therefore the inequality $(-1)^{r}\deg_\PP^{\mathbf{0}}(X) \ge 0$ follows, and we can conclude the proof by induction on $r$.
		Indeed, for any $1 \le i \le p$ and any general hyperplane $H \subset \PP_\kk^{m_i} \subset \PP_\kk^{m_1} \times_{\kk} \cdots \times_{\kk} \PP_\kk^{m_p}$, we have that $Y = X \cap H$ also has rational singularities (see the corresponding adaptation of Bertini's theorem in \cite[Remark 3.4.11(3)]{JOINS_INTERS}) and that $\deg_\PP^{\bn-\ee_i}(Y) = \deg_\PP^{\bn}(X)$ for all $\bn \in \NN^p$ with $\bn \ge \ee_i$.
	\end{proof}

	\section{Shellability and Hilbert polynomials}
	\label{sect_shellable}
	
	In this section, we deal with the Hilbert polynomial of a square-free monomial ideal whose corresponding simplicial complex is shellable. 
	For this family of ideals, we recover and strengthen the previous motivating results from \autoref{sect_motivation} (see \autoref{cor_pos_Marley} and \autoref{thm_Brion}).

	\begin{definition}
		Let $\Delta$ be a simplicial complex with facets $F_1,\ldots,F_k$.
		We say that $\Delta$ is \emph{shellable with order $F_1,\ldots,F_k$} if $\Delta$ is of pure dimension and $\langle F_1,\ldots,F_{i-1} \rangle \cap \langle F_i \rangle$ is pure of dimension $\dim(F_i)-1$  for all $2 \le i \le k$, where $\langle F_1,\ldots,F_{i-1} \rangle$ denotes the simplicial complex generated by $F_1,\ldots, F_{i-1}$.
	\end{definition}
	\begin{setup}
		Let $\kk$ be a field and $S = \kk[x_1,\ldots,x_n]$ be a standard $\NN^p$-graded polynomial ring. 
		Let $\PP = \multProj(S)$ be the corresponding multiprojective space. 
		For a square-free monomial ideal $J \subset S$, we denote by $\Delta(J)$ its Stanley-Reisner simplicial complex.
	\end{setup}

	First, we have a general formula for the Hilbert polynomial of a square-free monomial ideal; the following Inclusion-Exclusion polynomial  is useful to provide a simplified version of this formula.
	
	\begin{notation}
		Given a sequence of  prime ideals $\pp_1,\ldots,\pp_k \in \Spec(S)$ and a subset $\bbS \subseteq [k]$, we define the ideal $\pp_\bbS := \sum_{i \in \bbS} \pp_i$.
		If $\bbS = \emptyset$, we set $\pp_\bbS = S$.
		We further specify 
		$$
		\text{IE}(\{\pp_1,\ldots,\pp_k\})(\ttt) :=\sum_{\bbS\subseteq [k]} (-1)^{|\bbS|-1}P_{S/\pp_\bbS}(\ttt).
		$$
	\end{notation}
	
	\begin{lemma}
		\label{lem_general_IE}
		Let $J \subset S$ be a square-free monomial ideal, and consider the corresponding minimal primary decomposition $J = \pp_1 \cap \cdots \cap \pp_k$.
		Then we have the equality
		$$
		P_{S/J}(\ttt) \;= \; {\rm IE}(\{\pp_1,\ldots,\pp_k\})(\ttt).
		$$
	\end{lemma}
	\begin{proof}
		We proceed by induction on $k$.
		The case $k = 1$ is clear.
		We have the short exact sequence
		$$
		0 \rightarrow S/J \rightarrow S/\pp_1\cap\cdots\cap \pp_{k-1}\oplus S/\pp_k \rightarrow S/(\pp_1\cap\cdots\cap \pp_{k-1})+\pp_k \rightarrow 0.
		$$
		The additivity of Hilbert polynomials gives the equality 
		$$
		P_{S/J}(\ttt)= P_{S/\pp_{1}\cap\cdots\cap \pp_{k-1}}(\ttt)+ P_{S/\pp_k}(\ttt) - P_{S/(\pp_{1}\cap\cdots\cap \pp_{k-1})+\pp_k}(\ttt).
		$$
		Since each $\pp_i$ is generated by a collection of variables, we have the distributive equality 
		$$
		(\pp_1 \cap \cdots \cap \pp_{k-1}) + \pp_k \;=\; (\pp_1+\pp_k) \cap \cdots \cap (\pp_{k-1} + \pp_k).
		$$
		After applying the inductive hypothesis to the square-free monomial ideals $J_1 = \pp_{1}\cap\cdots\cap \pp_{k-1}$ and $J_2 = (\pp_1+\pp_k) \cap \cdots \cap (\pp_{k-1} + \pp_k)$, we obtain
		$$
		P_{S/J}(\ttt) \;=\; \sum_{\bbS \subseteq [k-1]} (-1)^{|\bbS|-1} P_{S/\pp_\bbS}(\ttt) + P_{S/\pp_k}(\ttt) + \sum_{\{k\} \subsetneq \bbS \subseteq [k]} (-1)^{|\bbS|-1} P_{S/\pp_\bbS}(\ttt),
		$$
		whence the result follows.
	\end{proof}

	\begin{remark}
		\label{rem_mon_prime}
		If $\pp \subset S$ is a monomial prime ideal, then the corresponding Hilbert polynomial is given by 
		$$
		P_{S/\pp}(\ttt) \;=\; \binom{t_1+n_1(\pp)-1}{n_1(\pp)-1} \cdots \binom{t_p + n_p(\pp)-1}{n_p(\pp)-1}
		$$
		where $n_i(\pp)$ equals the number of variables of degree $\ee_i$ not in $\pp$.
		In particular, $P_{S/\pp}(\ttt)$ has only one nonzero Hilbert coefficient.
	\end{remark}
	
	Under a shellable assumption, we have a significant improvement of the formula given in \autoref{lem_general_IE}.
	
	%
	%
	%

	\begin{proposition}\label{shell_prop}
		Let $J \subset S$ be a square-free monomial ideal, and consider the corresponding minimal primary decomposition $J = \pp_1 \cap \cdots \cap \pp_k$. 
		Suppose that 
		$\Delta(J)$ is shellable, and that the shelling order of the facets is given by $F_1,\ldots,F_k$
		where $F_i$ corresponds to $\pp_i$.
		Let $c$ be the common codimension of the primes $\pp_i$.
		Then, depending on this shelling order, we have the equality
		\begin{equation}\label{hilb_pol_sum}
			P_{S/J}(\ttt) \;=\; \sum_{j=1}^k \left(P_{S/\pp_j}(\ttt)-D_j(\ttt)\right),
		\end{equation}
		where $D_j(\ttt)={\rm IE}(\mathfrak{P}_j)(\ttt)$ and $\mathfrak{P}_j$ is the following sequence {\rm(}without repetitions{\rm)} of prime ideals
		$$
		\mathfrak{P}_j \,:= \, \lbrace \pp_{\ell}+\pp_j \mid \ 1\leq \ell\leq j-1 \text{ and } \codim(\pp_\ell+\pp_j)=c+1  \rbrace. 
		$$
		Furthermore,  for each $1\le j\le k$, there exists a set of  distinct variables $\{x_{\ell_{j,1}},\ldots, x_{\ell_{j,b_j}}\}\subset \pp_{[j-1]}\setminus \pp_j$ such that $\mathfrak{P}_j=\{(\pp_j,x_{\ell_{j,1}}),\ldots, (\pp_j,x_{\ell_{j,b_j}})\}$.
	\end{proposition}
	\begin{proof}
		Again, the proof follows by induction on $k$, and the base case $k=1$ is clear.
		The shellability assumption yields the equalities  
		$$
		J'  := (\pp_1\cap\cdots\cap\pp_{k-1}) + \pp_k 
		= (\pp_1+\pp_k) \cap \cdots \cap (\pp_{k-1} + \pp_{k}) \\
		= \bigcap_{\pp \in \mathfrak{P}_k} \pp. 
		$$
		Moreover, if $\codim(\pp_j + \pp_k)=c+1$, then the ideal $\pp_j + \pp_k$ equals the sum of $\pp_k$ and a variable not in $\pp_k$.
		The short exact sequence $
		0 \rightarrow S/J \rightarrow S/\pp_1\cap\cdots\cap \pp_{k-1}\oplus S/\pp_k \rightarrow S/J' \rightarrow 0
		$ and the induction hypothesis give the equality
		$$
		P_{S/J}(\ttt) \;=\; \sum_{j=1}^{k-1}   \left(P_{S/\pp_j}(\ttt)-D_j(\ttt)\right) + P_{S/\pp_k}(\ttt) - D_k(\ttt),
		$$
		and so the result follows.
	\end{proof}

	We discuss several important consequences of the formula given in \autoref{shell_prop}.
	
	\begin{corollary}
		\label{cor_positivity_shellable}
		Adopt the same assumptions and notation of \autoref{shell_prop}.
		Consider the corresponding closed subscheme $Z = \multProj(S/J) \subset \PP$.
		Then, for every $1\le j\le k$ and $\fJ \subset[b_j]$, we have $\codim(\pp_{j,\fJ})=\codim(J)+|\fJ|$ where $\pp_{j,\fJ}:=\sum_{i\in \bbS}(\pp_j,x_{\ell_{j,i}})$. 
		Thus,  \autoref{hilb_pol_sum} provides a cancellation-free expression for $P_{S/J}(\ttt)$ as a sum of the polynomials $P_{S/\pp_j}(\ttt)$ and  $(-1)^{|\fJ|}P_{S/\pp_{j,\fJ}}(\ttt)$ {\rm(}see \autoref{rem_mon_prime}{\rm)}.
		In particular,  $(-1)^{\dim(Z)-|\bn|}\deg^\bn(Z)\ge 0$ for all $\bn\in \NN^p$.
	\end{corollary}
	\begin{proof}
		Only the last statement requires a proof. 
		For that, we note that for each $\fJ \subset [b_j]$, we obtain  
		$$
		P_{S/\pp_{j,\fJ}}(\ttt) \;=\; \binom{t_1+n_1(\pp_{j,\fJ})-1}{n_1(\pp_{j,\fJ})-1} \cdots \binom{t_p + n_p(\pp_{j,\fJ})-1}{n_p(\pp_{j,\fJ})-1},\quad \text{and} 
		$$
		$|\fJ|=\codim(\pp_{j,\fJ})-\codim(J)=\left(\dim(S/J)-p\right)-\left(\dim(S/\pp_{j,\fJ})-p\right) = \dim(Z) - \sum_{i=1}^p(n_i(\pp_{j,\fJ})-1)$.
	\end{proof}

	\begin{corollary}
		\label{cor_Mobius_shellable_ideals}
		Let $J \subset S$ be a square-free monomial ideal such that  $\Delta(J)$ is shellable. 
		Consider the corresponding closed subscheme $Z = \multProj(S/J) \subset \PP$.
		Then we have the equality
		$$\sum_{\fn\in \NN^p}\deg_\PP^\bn(Z)=1.
		$$
	\end{corollary}
	\begin{proof}
		It suffices to show that $P_{S/J}(\mathbf{0}) = 1$.
		As in \autoref{shell_prop}, suppose  $\pp_1,\ldots,\pp_k$ is an ordering of the minimal primes such that the corresponding facets with this order give a shelling of $\Delta(J)$.
		We have that $P_{S/\pp_1}(\mathbf{0})=1$.
		On the other hand, for all $j \ge 2$, we have that 
		$$
		P_{S/\pp_j}(\mathbf{0}) - D_j(\mathbf{0}) \;=\; (1-1)^{b_j} \;=\; 0;
		$$
		this follows directly from the binomial formula.
		Thus the claimed equality follows from \autoref{hilb_pol_sum}.
	\end{proof}

	\begin{definition}
		We say that a  sequence of lattice points 
		$\ba_1, \ldots, \ba_u\in \NN^p$ is an {\it increasing lattice path} if for every $2 \le j\le u$  we have  $\ba_{j+1}-\ba_{j}= \be_i$ for some $i\in [p]$. 
	\end{definition}

	\begin{corollary}
		\label{cor_path_connected}
		Let $J \subset S$ be a square-free monomial ideal such that  $\Delta(J)$ is shellable.
		Consider the corresponding closed subscheme $Z = \multProj(S/J) \subset \PP$.
		Then for every  $\bn\in \hsupp_\PP(Z)$ there is an increasing lattice path from $\bn$ to a vector   $\ba\in \msupp_\PP(Z)$. 
	\end{corollary}
	\begin{proof}
		This is a consequence of \autoref{hilb_pol_sum} and \autoref{rem_mon_prime}.
		Indeed, if $\bn \in \hsupp_\PP(Z) \setminus \msupp_\PP(Z)$, then the corresponding term $\binom{t_1+n_1}{n_1}\cdots \binom{t_p+n_p}{n_p}$ is induced by a sum $D_j(\ttt)$ (which is an inclusion-exclusion sum of ideals given as the sum of $\pp_j$ and a variable not in $\pp_j$).
		Furthermore, in \autoref{cor_positivity_shellable} we pointed out that the expression of \autoref{hilb_pol_sum} is cancellation-free.
	\end{proof}

	\section{Stalactites, caves, and polymatroids}
	\label{sect_combinatorial}
	
	This section includes the definitions and important properties of our main combinatorial constructions: {\it stalactites} and {\it caves} (see \autoref{def_stal_neigh} and \autoref{def_caves}). 
	In the main theorem of this section, we show that every  cave is a  generalized polymatroid 
	(see \autoref{thm_glue_poly}). 	We begin with some general considerations about finite sets of lattice points.

	\begin{notation}\label{notation_sec_cave}
		Let $ \mathscr{A} \subset \mathbb{N}^p $ be a finite subset.
		We say that $\mathscr{A}$ is {\it homogeneous} if $|\fa|=|\bb|$ for every  $\fa,\bb\in \mathscr{A}$. 
		For an arbitrary $\mathscr{A}$, we define $\max(\mathscr{A}):= \max\{|\fa|\mid \fa\in \mathscr{A}\}$ and the {\it homogenization} of $\mathscr{A}$ by 
		\begin{equation*}\label{eq_homogenization}
			\mathscr{A}^{\text{hom}} \;:=\; \left\{ (\fa,\, \max(\A)-|\fa|) \in \mathbb{N}^{p+1} ~\mid~ \fa\in \mathscr{A}\right\}. 
		\end{equation*}
		The {\it top} of the set $\mathscr{A}$ is defined as 
			$	\mathscr{A}^{\ttop} \;:=\; \{ \mathbf{a} \in \mathscr{A} ~\mid~ |\fa|= \max(\mathscr{A})\}. $
		For each $ \bb \in \NN^d $, we define the {\it $\bb$-truncation} of $\mathscr{A}$ by
		$\mathscr{A}_\bb := \{ \mathbf{a} \in \mathscr{A} ~\mid~ \fa\gs \bb\}$.
\end{notation}

\begin{example}
Here is an example of a $\bb$-truncation with $\bb=(2,0,0)$.
On the left hand side of  \autoref{shift}, we depict a set where we have marked the points in the top.
On the right hand side, we truncated by  removing  the points whose first coordinate are less than or equal to 1 (cf. \autoref{sec_running_example}).

\begin{figure}[h]
	\includegraphics[scale=0.5]{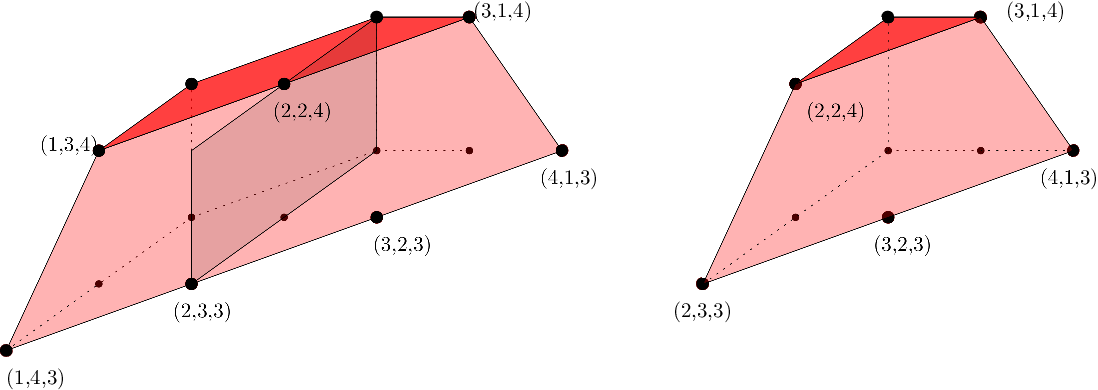}
	\caption{An example of a truncation.}
	\label{shift}
\end{figure}
\end{example}

We continue by recalling the definition of (discrete) base and generalized polymatroids.

\begin{definition}\label{definition_polymatroid}
A {\it base  polymatroid {\rm (}polymatroid{\rm )}} is a finite homogeneous subset $\mathscr{P} \subset \mathbb{N}^p$ satisfying the following property:
\begin{itemize}
	\item[] For every  $ \mathbf{u},\mathbf{v}\in \mathscr{P} $ and index $i\in [p]$ with $  [\mathbf{u}]_i >  [\mathbf{v}]_i $ there exists $ j \in [p]$ with $[\mathbf{u}]_j <  [\mathbf{v}]_j $  such that $ \mathbf{u} - \mathbf{e}_i + \mathbf{e}_j \in \mathscr{P} $.
\end{itemize}
\end{definition}

The following result by Herzog and Hibi provides a 
useful   property satisfied by base polymatroids.

\begin{theorem}[Symmetric Exchange, {\cite[Theorem 12.4.1]{HERZOG_HIBI}}]\label{prop_symmetric}
Let $\mathscr{P}$ be a base polymatroid. 
Then $\mathscr{P}$ satisfies the following property:
\begin{itemize}
	\item[] For every  $ \mathbf{u},\mathbf{v}\in \mathscr{P} $ and index $i\in [p]$ with $  [\mathbf{u}]_i >  [\mathbf{v}]_i $ there exists $ j \in [p]$ with $[\mathbf{u}]_j <  [\mathbf{v}]_j $  such that $ \mathbf{u} - \mathbf{e}_i + \mathbf{e}_j \in \mathscr{P} $  and $ \mathbf{v} - \mathbf{e}_j + \mathbf{e}_i \in \mathscr{P} $.
\end{itemize}

\end{theorem}

%

\begin{definition}\label{def_generalized_polymatroid}
A {\it generalized  polymatroid {\rm (}g-polymatroid{\rm )}} is a subset $ \mathscr{G} \subset \mathbb{N}^p $ such that $  \mathscr{G}^{\text{hom}}  $ is a   polymatroid.

The defining symmetric exchange property of   polymatroids translates to the following two properties of g-polymatroids:
\begin{itemize}
	\item[\underline{\rm Exchange}:]\label{exch} For every pair $ (\mathbf{u},\mathbf{v})\in \mathscr{G}\times \mathscr{G} $ and index  $ i\in [p] $ such that $  [\mathbf{u}]_i >  [\mathbf{v}]_i $ at least one of the following conditions holds:  
	\begin{enumerate}[\rm (1)]
		\item There exists $ j \in [p]$ with $  [\mathbf{u}]_j <  [\mathbf{v}]_j $  such that $\mathbf{u} - \mathbf{e}_i + \mathbf{e}_j $ and
		$ \mathbf{v} - \mathbf{e}_j + \mathbf{e}_i$ are both in $\mathscr{G}$. 
		\item  $|\fu|> |\fv|$, and $\mathbf{u} - \mathbf{e}_i$ and $\mathbf{v} + \mathbf{e}_i$ are both in  $\mathscr{G}$. 
	\end{enumerate}
	\item[\underline{\rm Expansion}:]\label{expa}  For every pair $ (\mathbf{u},\mathbf{v})\in \mathscr{G}\times\mathscr{G} $  with $  |\mathbf{u}|<  |\mathbf{v}| $, there exists an index $j \in[p]$ such that $[\mathbf{u}]_j <  [\mathbf{v}]_j$, and $\mathbf{u}+\mathbf{e}_j$ and $\mathbf{v}-\mathbf{e}_j$ are both in $\mathscr{G}$.
\end{itemize}
\end{definition}
\begin{remark}\label{rem_only_one}
We note that, by \autoref{definition_polymatroid}, in order to show a set is a g-polymatroid, it is enough to verify the Exchange and Expansion conditions for $\fu$.
\end{remark}

\begin{definition}\label{def_stal_neigh}	
For a given polymatroid $\mathscr{P}\subset \NN^p$  and $\fu\in \mathscr{P}$, we define the following related objects:
\begin{enumerate}[\rm (i)]
	\item 	 We say that $\bv\in \mathscr{P}$ is the {\it neighbor} of $\fu$ in direction $(-\ell, j)$ if $\fv=\fu-\be_{\ell}+\be_j$ for some $\ell, j\in [p]$. 
	\item Let $\ell_1,\ldots, \ell_m$ be distinct elements of $\supp(\fu)$, we define the {\it stalactite} $\St(\fu; \ell_1, \ldots, \ell_m)$ to be the following set
	$$
	\St(\fu; \ell_1, \ldots, \ell_m):=\left\{\fu-\sum_{i\in  \bbS}\be_{\ell_i}\mid \bbS \subseteq [m]\right\}.
	$$
	Here,  $\sum_{i\in  \bbS}\be_{\ell_i}=0$ if $\bbS=\emptyset$.
	\item Given $\mathscr{V}\subset \mathscr{P}$, we consider the set $\mathfrak{I}$ of all $\ell\in [p]$ such that there is a neighbor of $\fu$ in $\mathscr{V}$ in direction $(-\ell, j)$ for some $j\in [p]$. We define $\St(\fu; \mathscr{V}):= \St(\fu; \mathfrak{I}).$
\end{enumerate}

\end{definition}

We are now ready to define the concept of caves.

\begin{definition}\label{def_caves} 	
A finite set $\mathscr{C}\subset \NN^p$ is a {\it cave} if for every truncation $\mathscr{A}:=\C_{\bb}$ with $\bb\in\NN^p$ the following conditions hold:
\begin{enumerate}[\rm (a)]
	\item $\mathscr{A}^{\ttop}$ is a polymatroid.
	\item For a given  order of the elements in $[p]$, we let $\prec$ be the induced lexicographical (lex) order on $\NN^p$, and we sort the elements of $\mathscr{A}^{\ttop}$ as:
	$\fa_1\prec\fa_2\prec \cdots \prec \fa_{|\mathscr{A}^{\ttop}|}.
	$
	Then, for every possible order on $[p]$, we have 
	\begin{equation}\label{formula_cave}
	\mathscr{A}=\bigcup_{i=1}^{|\mathscr{A}^{\ttop}|} \St(\fa_i; \{\fa_1,\ldots, \fa_{i-1}\}).
	\end{equation}
	\item If $\bb\neq \mathbf{0}$, then $\mathscr{A}$ is a g-polymatroid.
\end{enumerate} 
\end{definition}

\begin{example}\label{exam_stal_sec}
Consider the cave $\mathscr{A}$ whose points on the top are:
$
\mathscr{A}^{\ttop}=\{134,143,224,233,314,323,413\}
$, where by simplicity we denote the vector $(n_1,n_2,n_3)$ as $n_1n_2n_3$.  
If we order the elements of the set $[3]$  as $1<2<3$, then the natural lex order on $\mathscr{A}^{\ttop}$ is the one given above. The corresponding stalactites from \autoref{formula_cave} are presented in the following table:

\medskip
{\small
	\begin{center}
		\begin{tabular}{ |c|c|c|c|c|c|c|c|} 
			\hline
			Point & 134 & 143 & 224 & 233 & 314 & 323 & 413 \\
			\hline
			Stalactite & $ 134  $ & $ 143,133  $ & $  224,124  $ & $  233,223,133,123  $ & $ 314,214  $ & $ 323,313,223,213  $ & $  413,313  $ \\
			\hline
		\end{tabular}
	\end{center}
}
\medskip
\noindent On the left hand side of \autoref{fig:two_estalactita}, we  observe these  stalactites. The tails of the arrows are the points in $\mathscr{A}^{\ttop}$ and the tips are points in the corresponding stalactites (see \autoref{formula_cave}). 
On the right hand side of  \autoref{fig:two_estalactita}, we observe  the  stalactites when the elements of the set $[3]$ are ordered as $3<1<2$. Notice how the stalactites change but at the end they amount to the same cave.

\begin{figure}[h]
	\centering
	\includegraphics[scale=0.5]{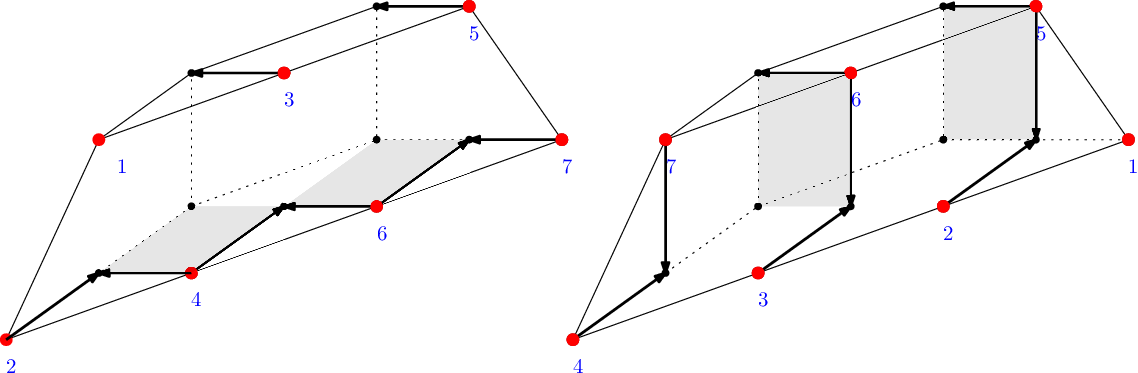}
	\caption{Two orderings of $[3]$ producing different stalactites.}
	\label{fig:two_estalactita}
\end{figure}
\end{example}	

\begin{notation}
\label{rem_lex_St}	
Henceforth, unless otherwise stated, we use the natural order on $[p]$. 
\end{notation}

Let $\C\subset \NN^p$ be a cave. The following facts, which follow from the definitions above, are used consistently in the proofs in this section.

\begin{fact}\label{fact_top}
For any  $\fa\in \mathscr{C}$ there exist $u:=\max(\mathscr{C})-|\fa|$ and  indices $\ell_1<\cdots<\ell_u$ in $[p]$ such that
$\fa^{\ttop}:=\fa+\be_{\ell_1}+\dots+\be_{\ell_u}\in \mathscr{C}^{\ttop}$ and $\fa\in \St(\fa^{\ttop}; \ell_1, \ldots, \ell_u)\subset \C$. 
Moreover, there exist $j_1,\ldots, j_u\in [p]$ with $\ell_i<j_i\ls p$ such that for each $i\in [u]$, the neighbor $\mathbf{b}_i:=\fa^{\ttop}-\be_{\ell_i}+\be_{j_i}$ of $\fa^{\ttop}$ belongs to $\mathscr{C}^{\ttop}$.
\end{fact}

\begin{fact}\label{fact_rem_last}
Continuing with the notation in \autoref{fact_top}, assume further that $p\in \supp(\fa)$. If $\fa^{\ttop}$ has a neighbor $\bb:=\fa^{\ttop}-\be_p+\be_j\in \C^{\ttop}$, then $\fa-\be_p+\be_j\in \C$. 
To see this, we apply \exc on $(\bb, \bb_i)$ at index $\ell_i$ to obtain a neighbor $\mathbf{c}_i:=\bb-\be_{\ell_i}+\be_{\alpha_i}$ for $\alpha_i\in\{j_i,p\}$ and we notice that $\fa-\be_p+\be_j\in \St(\bb;\ell_1, \ldots, \ell_u)=\St(\bb;\{\bc_1, \ldots, \bc_u\})\subset \C$.  
\end{fact}

\begin{fact}\label{fact_sup}
\exc and \expa hold for any $\fa,\bb\in \mathscr{C}$ with $\supp(\fa)\cap \supp(\bb)\neq\emptyset$. Indeed, this follows as $\C_{\be_k}$ is a g-polymatroid for any $k\in \supp(\fa)\cap \supp(\bb)$.
\end{fact}

In the following lemmas we prove some technical results about caves which are important for the proof of the main result of this section.

\begin{lemma}\label{lem_puntita}
Let $\C\subset \NN^p$ be a cave. 
Then for every $\mathbf{\fu}\in \C$ we have  $\fu+\be_i\in \C$ for every $i\in [p]$ such that $\C_{\fu+\be_i}\neq \emptyset$. 
\end{lemma}
\begin{proof}
By passing to the truncation $\C_\bu$, it suffices to consider the case $\fu=\mathbf{0}$. 
Set $r= \max(\C)$. Let $\ell_1, \ldots,\ell_r$ and $j_1,\ldots, j_r$ be the indices obtained by applying \autoref{fact_top} to $\fa=\mathbf{0}$. 
In particular, we get that $r \le p$, $\bn:= \be_{\ell_1}+\dots+\be_{\ell_r}\in \C^{\ttop}$, and $\St(\bn; \ell_1,\ldots, \ell_r)\subset \C$. It follows that $\be_{\ell_i}\in \C$ for every $i\in [r]$. 
Now, for each $k\in [r]$, we reorder $[p]$ as $(1,\ldots,\ell_k-1,\ell_k+1,\ldots,p,\ell_k)$ and use \autoref{fact_rem_last} with $\fa=\be_{\ell_k}$,  
$\fa^{\ttop} =\bn$, and $\bb=\bn-\be_{\ell_k}+\be_{j_k}$ to show $\be_{j_k}\in \C$ for each $k$.

To complete the proof, consider an index $i\in[p] \setminus \{\ell_1, \ldots, \ell_r, j_1,\dots,j_r\}$ 
such that  $\C_{\be_i}\neq \emptyset$.
Then, by \autoref{fact_top} there exists 
$\fu^{\ttop}\in\C^{\ttop}$ 
with $i\in\supp(\bu^\ttop)$.
From \excs on the pair ($\fu^{\ttop}$, $\fn$) at index $i$, we obtain a point 
$\fn'=\fn-\be_{\ell_m}+\be_i \in \C^{\ttop}$
for some $\ell_m$. 
If we reorder  $[p]$ as $(1, \ldots, i-1, i+1, \ldots, p, i)$, 
we obtain that $\fn'$ is smaller than $\fn$ in the lex order. 
Thus, if we replace $j_m$ by $i$  at the start of the  proof we  obtain that  $\ee_i \in \C$, finishing the proof.
\end{proof}

\begin{lemma}\label{claim}
	Let $ \mathscr{C} \subset\mathbb{N}^p$ be a cave. 
	Let $(\bu,\, \bv)\in \mathscr{C}\times \C$  be such that $\supp(\bu)\cap \supp(\bv)=\emptyset$ and let $i\in \supp(\bu)$. 
	If $ |\fv|\ge2 $, then either $\bv-\be_j+\be_i \in \mathscr{C}$ for some $j\in\supp(\bv)$ or $\bv+\be_i\in \mathscr{C}$. 
\end{lemma}
\begin{proof}
	Due to \autoref{fact_sup}, it is enough to show  that there exists $\bv'\in\mathscr{C}^{\text{top}}$ such that $i\in\supp(\bv')$ and $\supp(\bv)\cap\supp(\bv')\neq\emptyset$; indeed, \excs on $(\bv',\bv)$ at index $i$ yields the result.
	By \autoref{fact_top}, there exists $\fv^{\ttop}\in \C^{\ttop}$ such that $\fv\ls \fv^{\ttop}$. 
	If $i\in \supp(\fv^{\ttop})$, we can take $\fv'=\fv^{\ttop}$. 
	Otherwise, consider $\fu^{\ttop}\in \C^{\ttop}$ such that $\fu\ls \fu^{\ttop}$. 
	By \excs  on $(\fu^{\ttop},\, \fv^{\ttop})$ at index $i$, we obtain $\fv^{\ttop}-\be_j+\be_i\in \C^{\ttop}$ for some $ j\in \supp(\fv^{\ttop})$. 
	Moreover, $\supp(\bv)\cap\supp(\fv^{\ttop}-\be_j+\be_i)\neq \emptyset$ since $|\bv|\ge 2$.
	Therefore, we can take $\fv'=\fv^{\ttop}-\be_j+\be_i$.
\end{proof}

\begin{lemma}[\expa for caves]
	\label{lem_expansion}
	Let $\mathscr{C} \subset \NN^p$ be a cave and $(\bu, \bv) \in \mathscr{C} \times \mathscr{C}$ such that $|\bu| < |\bv|$.
	Then there exists $j \in [p]$ such that $[\bu]_j < [\bv]_j$ and $\bu + \ee_j \in \C$ {\rm(}see \autoref{rem_only_one}{\rm)}.
\end{lemma}
\begin{proof}
	If $\bu=\mathbf{0}$, the result  follows from \autoref{lem_puntita}.  
	Thus, we assume $|\bu| \ge 1$ and $|\bv| \ge 2$.
	Due to \autoref{fact_sup}, we may also assume $\supp(\bu)\cap \supp(\bv)=\emptyset$.
	Let $i \in \supp(\bv)$.
	By applying \autoref{claim} to the pair $ (\bu , \bv) $  and the index $ i \in \supp(\bu)$, we obtain that either $ \mathbf{v}':=\bv-\be_\ell+\be_i \in \C$  or $\mathbf{v}':=\bv+\be_i \in \C$ for some $\ell\in \supp(\fv)$. 
	In either case $i\in \supp(\fv')\cap\supp(\fu)$.
	 Notice $|\fv'|\gs |\fv|>|\fu|$ and then by \autoref{fact_sup} we apply \expa on $(\fu,\, \fv')$ to conclude that $[\fu]_j<[\fv']_j$ for some $j\in [p]$ and $\fu+\fe_j\in  \C$. 
	Since $i\not\in \supp(\fv)$, we must have $j\neq i$ and then $[\fv']_j\ls [\fv]_j$, which finishes the proof.
\end{proof}

\begin{lemma}[\excs for caves]
	\label{lem_exchange}
		Let $\mathscr{C} \subset \NN^p$ be a cave,  $(\bu, \bv) \in \mathscr{C} \times \mathscr{C}$, and $i \in [p]$ be an index 
		such that $[\bu]_i > [\bv]_i$. 
		Then at least one of the following conditions holds:
		\begin{enumerate}[\rm(1)]
			\item There exists $ j \in [p]$ with $  [\mathbf{u}]_j <  [\mathbf{v}]_j $  such that $\mathbf{u} - \mathbf{e}_i + \mathbf{e}_j \in \mathscr{C}$. 
			\item  $|\fu|> |\fv|$ and $\mathbf{u} - \mathbf{e}_i \in \mathscr{C}$. 
		\end{enumerate} 
	{\rm(}See \autoref{rem_only_one}.{\rm)}
\end{lemma}
\begin{proof}
	By \autoref{fact_sup}, we may assume $\supp(\bu)\cap \supp(\bv)=\emptyset$.
	After reordering the elements in $[p]$, we can assume $i=p$.
	Let $u = \max(\mathscr{C})-|\bu|$ and $v = \max(\mathscr{C})-|\bv|$.
	From \autoref{fact_top}, we choose top elements $\bu^\ttop = \bu + \be_{\ell_1}+\dots+\be_{\ell_{u}} \in \mathscr{C}^\ttop$ and $\bv^\ttop = \bv + \fe_{h_1}+\cdots+\fe_{h_v} \in \mathscr{C}^\ttop$.
	By applying \excs on $(\bu^\ttop, \bv^\ttop)$ at index $p$, we get $j \in [p]$ such that $[\bu^\ttop]_j<[\bv^\ttop]_j$ and $\bu^\ttop-\ee_p+\ee_j \in \mathscr{C}^\ttop$.
	From \autoref{fact_rem_last}, we obtain $\bu-\ee_p+\ee_j \in \mathscr{C}$.
	If $j \in \supp(\bv)$, then condition (1) holds. 
	
	For the rest of the proof, we assume $j \notin \supp(\bv)$, and so $[\bv^\ttop]_j = 1$ and $[\bu^\ttop]_j = 0$.
	Since $j \not\in \{\ell_1,\ldots,\ell_u, p\}$, after reordering $[p]$ as $(1,\ldots,j-1,j+1, \ldots,p, j)$ and considering the respective lex order, the point $\bu^\ttop$ keeps the neighbors in directions $(-\ell_1,j_1),\ldots,(-\ell_u,j_u)$ and gets the neighbor $\bu^\ttop-\ee_p+\ee_j$ in direction $(-p,j)$.
	It follows by definition that 
	$$
	\St(\bu^{\ttop}; \ell_1, \ldots, \ell_u,p) \;\subset\; \C.
	$$
	In particular, $\bu-\ee_p \in \C$.
	If $|\bu| > |\bv|$, then condition (2) holds. 
	Therefore we assume $|\bu| \le |\bv|$.
	Since $|\bu-\ee_p| < |\bv|$, \autoref{lem_expansion} yields an index $j' \in \supp(\bv)$ such that $\bu-\ee_p+\ee_{j'} \in \C$, and so condition (1) holds. 
	This concludes the proof.
\end{proof}

The following is the main theorem of this section. 

\begin{theorem}[Gluing g-polymatroids]\label{thm_glue_poly}
Every cave $ \mathscr{C} \subset \mathbb{N}^p $ is a g-polymatroid. 
\end{theorem}
\begin{proof}
	The result follows by combining \autoref{rem_only_one}, \autoref{lem_expansion} and \autoref{lem_exchange}.
\end{proof}

In light of \autoref{thm_glue_poly} we end the section with the following question: 

\begin{question}
Is every g-polymatroid a cave?	
\end{question}

	\section{Results on multiplicity-free varieties}
	\label{sect_results_mult_free}

	In this section, we prove several results on multiplicity-free varieties. 
	These results form the core of our applications. 	Throughout this section, we utilize the following setup.

	\begin{definition}\label{def_multideg}	
		Let $X \subset \PP = \PP_\kk^{m_1} \times_\kk \cdots \times_\kk \PP_\kk^{m_p}$ be a variety over an arbitrary field $\kk$.
		We say that $X$ is  \emph{multiplicity-free} if all its multidegrees are at most one, i.e., $\deg_\PP^\bn(X) \in \{0, 1\}$ for all $\bn \in \NN^p$ with $|\bn| = \dim(X)$. 	Let  $\bm = (m_1,\ldots,m_p) \in \ZZ_+^p$  and 
		$$
		S = \kk\left[x_{i,j} \mid 1 \le i \le p, 0 \le j \le m_i \right]
		$$ 
		be the standard $\NN^p$-graded polynomial ring with $\deg(x_{i,j}) = \ee_i \in \NN^p$ for every $i,j$. 
		We note that  
		$$
		\multProj(S) = \PP := \PP_\kk^{m_1} \times_\kk \cdots \times_\kk \PP_\kk^{m_p}.
		$$
	\end{definition}
	
	Our approach on multiplicity-free varieties comes from the general fact that we can degenerate these varieties to a reduced union of certain multiprojective spaces (and these multiprojective spaces are determined by the multidegrees).
	In algebraic terms, this says that the multigraded generic initial ideals are square-free.
	This follows from an important theorem of Brion \cite{BRION_MULT_FREE}.
	
	We now introduce some notation regarding multigraded generic initial ideals. 

	\begin{setup}
		\label{setup_gin}
		Let  $B_{m_i+1}(\kk)$ be the subgroup of $\text{GL}_{m_i+1}(\kk)$  of upper triangular invertible matrices.   
		The \emph{Borel subgroup of} $G := \text{GL}_{m_1+1}(\kk) \times \cdots \times \text{GL}_{m_p+1}(\kk)$ is the one given by 
		$$
		B := B_{m_1+1}(\kk) \times \cdots \times B_{m_p+1}(\kk)\subset G.
		$$  
		We consider the natural action of $G$ on $S$. Let $>$ be a monomial order on $S$ satisfying $x_{i,0} > x_{i,1} > \cdots > x_{i,m_i}$ for all $i \in [p]$. 
	\end{setup}

	We now recall the definition of multigraded generic initial ideals;~see \cite[\S15.9]{EISEN_COMM} for the single graded~case.  
	
	\begin{definition}
		Let $J \subset S$ be an $S$-homogeneous ideal. If $\kk$ is infinite, we define the
		 \emph{multigraded generic initial ideal} $\gin_>(J)$ of $J$ with respect to $>$ as the ideal $\init_>(g(J))$, where $g$ belongs to a Zariski dense open subset $\mathcal{U} \subset G$.
	\end{definition}
	
	An $S$-homogeneous ideal $J \subset S$ is said to be \emph{Borel-fixed} if $g(J) = J$ for all $g \in B$.
	As in the single graded setting, we have that  $\gin_>(J)$ is Borel-fixed (see \cite[Theorem 15.20]{EISEN_COMM}). 	
	
	\begin{remark}\label{rem_borel_primes}
		\label{rem_mult_free_varieties}
		The properties of Borel-fixed ideals (see~\cite[Chapter~15]{EISEN_COMM}) extend to the multigraded setting:
		\begin{enumerate}[\rm (i)]
			\item A Borel-fixed prime ideal in $S$ is of the form
			$
			\pp_\mathbf{a} \,:=\, \left(x_{i,j} \mid 1 \le i \le p \text{ and } 0 \le j < a_i\right) \subset S
			$
			for some $\mathbf{a} = (a_1,\ldots,a_p) \in \NN^p$ (see \cite[Lemma 3.1]{CDNG_MINORS}).
			\item An $S$-homogeneous Borel-fixed ideal $J \subset S$ is monomial  and all of its associated primes are also Borel-fixed (see \cite[Lemma 3.2]{CDNG_MINORS}).
		\end{enumerate}
	\end{remark}

	In the following theorem we gather some important properties of multiplicity-free varieties. 
	
	\begin{theorem}[\cite{BRION_MULT_FREE}, \cite{caminata2022multidegrees}]
		\label{thm_gin_mult_free}
		Assume \autoref{setup_gin}.
		Let $X \subset \PP = \PP_\kk^{m_1} \times_\kk \cdots \times_\kk \PP_\kk^{m_p}$ be a multiplicity-free variety and $\fP \subset S$ be the corresponding $\NN^p$-graded prime ideal.
		Then the following statements hold:
		\begin{enumerate}[\rm (i)]
			\item $S/\fP$ is Cohen-Macaulay.
			\item $S/\fP$ is a normal domain.
			\item  {\rm (}$\kk$ infinite{\rm )} $\gin_>(\fP)$ is a square-free monomial ideal for any monomial order $>$ on $S$. 
			\item {\rm (}$\kk$ infinite{\rm )} The simplicial complex $\Delta$ corresponding to $\gin_>(\fP)$ is shellable.
		\end{enumerate}
	\end{theorem}
	\begin{proof}
		For parts (i),(ii) and (iii), see \cite[Theorem 6.9]{caminata2022multidegrees}.
		Part (iv) follows from \cite[Remark 5.7]{caminata2022multidegrees}.
	\end{proof}

	\begin{remark}[{\cite[Remark 5.10]{caminata2022multidegrees}}]\label{rem_corr_prime_mult}
		Let $X \subset \PP = \PP_\kk^{m_1} \times_\kk \cdots \times_\kk \PP_\kk^{m_p}$ be a multiplicity-free variety with $\kk$ infinite and $\fP \subset S$ be the corresponding $\NN^p$-graded prime ideal.
		We have that $\gin_>(\fP)$ is square-free.
		Then there is a bijection between the positive multidegrees of $X$ and the (minimal) associated primes of $\gin_>(\fP)$.
		More precisely, for any $\bn \in \NN^p$ with $|\bn| = \codim(X, \PP)$, we have that $c_\bn(X) > 0$ if and only if $\pp_{\bn} \subset S$ is an associated prime of $\gin_>(\fP)$ (see \autoref{rem_borel_primes}(i)), and by \cite[Theorem 2.8(ii)]{cidruiz2021mixed}, we also should have $\bn \le \bm$.		
		Equivalently, we obtain that $\bn \in \msupp_\PP(X)$ if and only if $\pp_{\bm - \bn}$ is an associated prime of $\gin_>(\fP)$.
	\end{remark}

	The following lemma shows that for multiplicity-free varieties, Hilbert functions and polynomials agree everywhere in $\fv\in\NN^p$. 
	
	\begin{lemma}
		\label{lem_eq_Hilb_pol_funct_mult_free}
		Let $X \subset \PP = \PP_\kk^{m_1} \times_\kk \cdots \times_\kk \PP_\kk^{m_p}$ be a multiplicity-free  variety and $\fP \subset S$ be the corresponding $\NN^p$-graded prime ideal. 
		Then we have the equality 
		$$
		P_X(\fv) \;= \; \dim_\kk\left(\left[S/\fP\right]_\fv\right)
		$$
		for all $\fv \in \NN^ p$.
	\end{lemma}
	\begin{proof}
		After possibly extending the field, we can assume $\kk$ is infinite. Choose a monomial order $>$ and the corresponding generic initial ideal $J =\gin_>(\fP) \subset S$.
		We have that $P_{X}(\ttt) = P_{S/J}(\ttt)$ (see, e.g., \cite[Theorem 15.3]{EISEN_COMM}).
		By \autoref{thm_gin_mult_free}, we know that $J$ is a square-free monomial ideal.
		From \autoref{lem_general_IE}, we can compute the Hilbert polynomial of $S/J$ in terms of the Hilbert polynomials of sums 
		of the minimal primes of $J$.
		From \autoref{rem_mult_free_varieties}(i), we know that the minimal primes of $J$ are of the form $\pp_{\ba_1},\ldots, \pp_{\ba_k}$.
		Since the 
		sum of minimal primes
		$
		\pp_{\ba_{i_1}} + \cdots + \pp_{\ba_{i_\ell}} \;=\; \pp_{\max\{\ba_{i_1}, \ldots, \ba_{i_\ell}\}} 
		$
		gives a relevant monomial prime ideal (see \autoref{rem_corr_prime_mult}), it follows that 
		$$
		P_{S/(\pp_{\ba_{i_1}} + \cdots + \pp_{\ba_{i_\ell}})}(\fv) = \dim_\kk\left({\big[S/(\pp_{\ba_{i_1}} + \cdots + \pp_{\ba_{i_\ell}})\big]}_\fv\right) 
		$$
		for all $\fv \in \NN^p$.
		This completes the proof of the lemma.
	\end{proof}

The following lemma shows that for a multiplicity-free variety, the order of the facets of $\Delta\left(\gin_>(\fP)\right)$ induced by the lex order on the multidegrees  is a shelling. We note that, when considering the lex order, we can do so with respect to any ordering of the elements in $[p]$. 

\begin{lemma}\label{lemma_Lex}
	{\rm(}$\kk$ infinite{\rm)} Let $X \subset \PP = \PP_\kk^{m_1} \times_\kk \cdots \times_\kk \PP_\kk^{m_p}$ 	 be a multiplicity-free  variety  and $\fP \subset S$ be the corresponding $\NN^p$-graded prime ideal. 
	Let $J=\gin_>(\fP) \subset S$, and  	
	for each $\bn\in \Msupp_\PP(X)$, let $F_\bn$  be the corresponding facet of $\Delta(J)$, i.e.,
	$$
	F_\bn:=\{x_{1,m_1-n_1},\ldots,x_{1,m_1}\}\cup \cdots\cup \{x_{p,m_p-n_p},\ldots,x_{p,m_p}\}.
	$$  
	Then, the lex order on $\msupp_\PP(X) \subset \NN^p$ induces a shelling on $\Delta(J)$.  
\end{lemma}
\begin{proof}
	The set  $\mathscr{P}:= \Msupp_\PP(X) \subset \NN^p$ is a 
	 polymatroid  by \autoref{thm_pos_multdeg} 
	 (see \autoref{definition_polymatroid}). 
	 Denote by $\prec$ the lexicographical order on $\NN^p$.
	Fix a point $\fn=(n_1,\ldots, n_{p})\in\mathscr{P}$, by induction, it suffices to show that for any $\fs=(s_1,\dots,s_p)\in \mathscr{P}$ with $\fs \prec \fn $ and $|F_{\fn}\cap F_{\fs}|<|F_\fn|-1$ there exists $\fv\in \mathscr{P}$ such that $\fv\prec \fn$ and $F_{\fs}\cap F_{\fn} \subsetneq F_{\fv}\cap F_{\fn} $. 
	We note that, $F_\fs\cap F_\fn=F_{\min\{\fs,\fn\}}$ where $\min\{\fs,\fn\}:=(\min\{s_1,n_1\}, \ldots, \min\{s_p,n_p\})$.
	
	Let $i\in [p]$ be the smallest index such that $s_i\neq n_i$, since $\fs\prec\fn$ we have by definition that $s_i <n_i$. 
	Thus, 
	there exists $j\in [p]$ such that $i<j$, $n_j<s_j$, and $\fs':=\fs+\be_i-\be_j\in \mathscr{P}$.
	Since $\min\{\fs',\fn\}=\min\{\fs,\fn\}+\be_i$, it follows that $F_{\fs}\cap F_{\fn} \subsetneq F_{\fs'}\cap F_{\fn} $. 
	If $\fs'\prec \fn$, the proof is finished, otherwise we must have $n_i=s_i+1$, and we assume this is the case.
	
	If $s_{i'}<n_{i'}$ for some other $i'> i$, then there exists $j'\in [p]$ such that $n_{j'}<s_{j'}$ 
	and $\fu:=\fs+\be_{i'}-\be_{j'}$ satisfies $\fu\prec\fn$ and $F_{\fs}\cap F_{\fn}\subsetneq F_{\fu}\cap F_{\fn}$, 
	and again the proof is finished.
	Finally, if no such $i'$ exists, we must have $n_j+1=s_j$  and $s_k=n_k$ for every   $k\neq i,j$.
	It would follow that $|F_{\fs}\cap F_{\fn}|=|F_{\fn}|-1$, contradicting the hypothesis $|F_{\fs}\cap F_{\fn}|<|F_{\fn}|-1$.
\end{proof}

	In the following theorem we show that if one knows the multidegrees of a multiplicity-free variety, then one can determine the whole Hilbert polynomial. Moreover, we  provide an algorithmic way to do so.
		
		\begin{theorem}[Hilbert polynomials are determined by multidegrees]
		\label{lem_stalactite_mult_free}
		Let $X \subset \PP = \PP_\kk^{m_1} \times_{\kk} \cdots \times_{\kk} \PP_\kk^{m_p}$ be a multiplicity-free variety.
		 Let $d=|\msupp_\PP(X)|$ and 
	 	$\fn_1\prec\fn_2\prec \cdots \prec \fn_{d}
		$
		be the elements of $\msupp_\PP(X)$ sorted with respect to the lex order. Then
			\begin{equation*}
			\hsupp_\PP(X)=\bigcup_{i=1}^{d} \St(\fn_i; \{\fn_1,\ldots, \fn_{i-1}\})\;\;\text{	    (see \autoref{def_stal_neigh}). }
		\end{equation*}
	Moreover, for each $\bn\in \hsupp_\PP(X)$ we have  $$\deg_\PP^\bn(X)=(-1)^{\dim(X)-|\bn|}\big|\left\{i\in [d]\mid \bn\in \St(\fn_i; \{\fn_1,\ldots, \fn_{i-1}\right\})\}\big|.$$
%
	\end{theorem}
	\begin{proof}
		After possibly extending the field, we can assume $\kk$ is infinite. 
		Let $\fP \subset S$ be the corresponding $\NN^p$-graded prime ideal, and consider a generic initial ideal $J = \gin_>(\fP)$ with respect to some monomial order $>$.
		Due to \autoref{thm_gin_mult_free}, the monomial ideal $J$ is square-free and the corresponding simplicial complex $\Delta(J)$ is shellable.
		It remains to provide an improvement of \autoref{cor_path_connected}. 
		
		Let $\pp_{\ba_1},\ldots, \pp_{\ba_d}$ be the minimal primes of $J$ (see \autoref{rem_mult_free_varieties}(i)) where $\pp_{\ba_i}$ corresponds to $\bn_i \in \msupp_\PP(X)$ (see \autoref{rem_corr_prime_mult}).
		By \autoref{lemma_Lex}, the  primes  are sorted increasingly  with respect to a shelling order. 
		If $\bn \in \hsupp_\PP(X) \setminus \msupp_\PP(X)$, then the corresponding term $\binom{t_1+n_1}{n_1}\cdots \binom{t_p+n_p}{n_p}$ is induced by one of the  $D_j(\ttt)$  (see \autoref{shell_prop}). 
		As each $\pp_{\ba_i}$ is Borel-fixed, in the description given in \autoref{shell_prop}, we would obtain that in the set $\mathfrak{P}_j=\{(\pp_j,x_{\ell_{j,1}}),\ldots, (\pp_j,x_{\ell_{j,b_j}})\}$ all the new variables $x_{\ell_{j,1}}, \ldots,x_{\ell_{j,b_j}}$ have different degrees.
		The result of the theorem 
		now follows from the formula in \autoref{hilb_pol_sum}.
	\end{proof}

\begin{example}\label{exam_mult_free_sec}
	Let $I_w\subset S$ be the ideal in our running example \autoref{sec_running_example}. 
	The positive multidegrees of the variety $X=\multProj(S/I_w) \subset \big(\PP_\kk^4\big)^5$  can be identified with the set of points from \autoref{exam_stal_sec}, i.e., 
	$$\big\{134,\,143,\,224,\,233,\,314,\,323,\,413\big\}.$$
Therefore, by \autoref{lem_stalactite_mult_free} we can build the Hilbert polynomial of $X$ from the stalactites in 	 \autoref{exam_stal_sec}. 
We obtain the expression
	\begin{align*}\label{eq:H-example}
	P_X(\ttt)=&	\langle\ttt\rangle^{13444}+\langle\ttt\rangle^{14344}+\langle\ttt\rangle^{22444}+\langle\ttt\rangle^{23344}+\langle\ttt\rangle^{31444}+\langle\ttt\rangle^{32344}+\langle\ttt\rangle^{41344}\\
	&-\langle\ttt\rangle^{12444}-2\langle\ttt\rangle^{13344}-\langle\ttt\rangle^{21444}-2\langle\ttt\rangle^{22344}-2\langle\ttt\rangle^{31344}+\langle\ttt\rangle^{12344}+\langle\ttt\rangle^{21344},
	\end{align*}
where $\langle\ttt\rangle^{n_1n_2n_3n_4n_5}:= \binom{t_1+n_1}{n_1}
\binom{t_2+n_2}{n_2}\binom{t_3+n_3}{n_3}\binom{t_4+n_4}{n_4}\binom{t_5+n_5}{n_5}.$
\end{example}

	We now introduce a generic version of Bertini's theorem that will be enough for our purposes.
	
	\begin{remark}[{\cite[Corollary 1.5.9, Corollary 1.5.10]{JOINS_INTERS}}]
		Let $X \subset \PP = \PP_\kk^{m_1} \times_\kk \cdots \times_\kk \PP_\kk^{m_p}$ be a subvariety.	
		Fix $i \in [p]$.
		Consider the purely transcendental field extension $\LL := \kk(z_0,\ldots,z_{m_i})$.
		Let $S_\LL := S \otimes_\kk \LL$, $\PP_\LL := \PP \otimes_\kk \LL = \multProj(S_\LL)$, and $X_\LL = X \otimes_\kk \LL \subset \PP_\LL$. We have that $X_\LL$ is irreducible and reduced.
We say that $y_i := z_0x_{i,0} + \cdots + z_{m_i}x_{i,m_i}$  is the \emph{generic element} of $\left[S_\LL\right]_{\ee_i}$ and that $H_i = V(y_i) \subset \PP_\LL^{m_i} \subset \PP_\LL$ is the \emph{generic hyperplane} in the $i$-th component.
		If $\dim(X) \ge 1$, then $X_\LL \cap H_i \subset \PP_\LL$ is also irreducible and reduced.
	\end{remark}

	The next proposition includes  further important properties of multiplicity-free varieties.
	
	\begin{proposition}\label{prop_mult_free_statements}
		Let $X \subset \PP = \PP_\kk^{m_1} \times_{\kk} \cdots \times_{\kk} \PP_\kk^{m_p}$ be a multiplicity-free variety.
		Then the following statements hold:
		\begin{enumerate}[\rm (i)]
			\item {\rm[Behavior under a cut with a generic hyperplane].} Let $i \in [p]$ and $H_i \subset \PP_\LL^{m_i} \subset \PP_\LL$ be the generic hyperplane in the $i$-th component.
			The intersection $X_\LL \cap H_i$ is a multiplicity-free variety of dimension $\dim(X)-1$.
			Furthermore, $\deg_{\PP_\LL}^{\bn -\ee_i}(X_\LL \cap H_i) = \deg_\PP^{\bn}(X)$ for all $\bn = (n_1,\ldots,n_p) \in \NN^p$ and $n_i\ge 1$.
			\item {\rm[Behavior under a projection].} Let $\fJ = \{j_1,\ldots,j_k\} \subset [p]$ be a subset, and set $\PP' = \PP_{\kk}^{m_{j_1}} \times_\kk \cdots \times_\kk \PP_{\kk}^{m_{j_k}}$.
			The corresponding projection $\Pi_{\fJ}(X) \subset \PP'$ is a multiplicity-free variety.
			Furthermore, for any $\bn = (n_1,\ldots,n_k) \in \NN^k$ the corresponding Hilbert coefficient is given by
			$$
			\deg_{\PP'}^{\bn}\big(\Pi_{\fJ}(X)\big) \;=\; \sum_{ \bw \in \mathscr{W}} \deg_\PP^{\bw}(X),
			$$			
			where $\mathscr{W} = \lbrace \bw = (w_1,\ldots,w_p) \in \NN^p  \mid w_{j_1} = n_{1}, \ldots, w_{j_k} = n_{k} \rbrace$.
			
			\item {\rm[M\"obius-like behavior of Hilbert coefficients].}
			Let $\bn \in \NN^p$ be such that there exists $\bw \in \msupp_\PP(X)$ with $\bw \ge \bn$.
			The following equality holds
			$$
			\sum_{\bw \ge \bn} \deg_\PP^\bw(X) \;=\; 1.
			$$			
			\item For all $i \in [p]$, we have the following equality
			$$
			\min\big\{n_i \mid \bn = (n_1,\ldots,n_p) \in \hsupp_\PP(X)\big\}  \;=\; \min\big\{n_i \mid \bn = (n_1,\ldots,n_p) \in \msupp_\PP(X)\big\}. 
			$$
			\item Let $\fJ = \{j_1,\ldots,j_k\} \subset [p]$ be a subset such that $\dim\left(\Pi_{[p] \setminus \fJ}(X)\right)=0$.
			Consider the corresponding projection $X' = \Pi_\fJ(X) \subset \PP_\kk^{m_{j_1}}\times_{\kk} \cdots \times_\kk \PP_\kk^{m_{j_k}}$.
			Then we have that 
			$$
			X \cong X' \quad \text{ and } \quad P_X(\ttt) =	P_{X'}(t_{j_1},\ldots,t_{j_k}).	
			$$
		\end{enumerate}
	\end{proposition}
	\begin{proof}
		Let $\fP \subset S$ be the  $\NN^p$-graded prime ideal corresponding to $X$.
		
		(i) This is a basic result; see \cite[Lemma 3.10]{POSITIVITY} and \cite{cidruiz2021mixed}.
		
		(ii) Let $R = S/\fP$ and $Y = \Pi_\fJ(X) \subset \PP'$.
		The fact that $Y$ is multiplicity-free follows from \cite[Theorem~C]{caminata2022multidegrees}.
		The multihomogeneous coordinate ring of $Y$ is given by the standard $\NN^k$-graded algebra
		$$
		R_{(\fJ)} = \bigoplus_{\substack{i_1\ge 0,\ldots, i_p\ge 0\\ i_{j} = 0 \text{ if } j \not\in \fJ}} {\left[R\right]}_{(i_1,\ldots,i_p)};
		$$
		in other words, $Y = \multProj\left(R_{(\fJ)}\right)$.
		Due to \autoref{lem_eq_Hilb_pol_funct_mult_free}, we obtain the equalities
		\begin{align*}
			P_Y(\nu_{j_1},\ldots,\nu_{j_k}) &\;=\; \dim_\kk\left(\left[R_{(\fJ)}\right]_{(\nu_{j_1},\ldots,\nu_{j_k})}\right)
			\;=\; \dim_\kk\left(\left[R\right]_{\fv}\right) 
			\;=\; P_X(\fv)
		\end{align*}
		for all $\fv = (\nu_1,\ldots,\nu_p) \in \NN^p$ such that $\nu_j = 0$ when $j \notin \fJ$.
		Therefore,  $P_Y(t_{j_1},\ldots,t_{j_k})$ is obtained from the polynomial $P_X(t_1,\ldots,t_p)$ by substituting $t_j= 0$ for $j \notin \fJ$.
		The claimed formula for the multidegrees of $Y$ follows by comparing coefficients.
		
		(iii) 
		Let $\bn \in (n_1,\ldots,n_p) \in \NN^p$ and suppose there is some $\bw \in \msupp(X)$ such that $\bw \ge \bn$.
		For each $i \in [p]$, choose a sequence of generic hyperplanes $H_{i,1},\ldots,H_{i,n_i} \subset \PP_\LL^{m_i} \subset \PP_\LL$.
		Set $Z = X_\LL \cap \left(\bigcap_{i,j} H_{i,j}\right)$.
		By utilizing part (i), we obtain that 
		$$
		\deg_\PP^\bw(X) \;=\; \deg_{\PP_\LL}^{\bw-\bn}(Z)
		$$
		for all $\bw \ge \bn$.
		Moreover, we know that $Y$ is not empty due to the existence of a positive multidegree $\bw \in \msupp_\PP(X)$ such that $\bw \ge \bn$.
		
		Therefore, we may translate the point $\bn$ to the origin $\mathbf{0} \in \NN^p$, and we do so.
		Let $Y = \Pi_p(X)$ be the projection to the last component. 
		As $Y$ is a projective variety of degree $1$, it must follow that $Y \cong \PP_{\kk}^l \subset \PP_{\kk}^{m_p}$ for some $l \ge 0$.
		By part (ii), we obtain the equation
		$$
		\sum_{n_1,\ldots,n_{p-1} \ge 0} \deg_\PP^{(n_1,\ldots,n_{p-1}, n_p)}(X)  \;= \; \deg_{\PP_\kk^{m_p}}^{n_p}(Y)  \;=\; \begin{cases}
			1 & \text{ if $n_p = l$}\\
			0 & \text{ otherwise.}
		\end{cases} 
		$$
		Summing up over $n_p \ge 0$ yields the M\"obius-like behavior of Hilbert coefficients.
		Alternatively, we may conclude the proof by utilizing \autoref{thm_gin_mult_free}(iv) and \autoref{cor_Mobius_shellable_ideals}.

		(iv) This follows from \autoref{lem_stalactite_mult_free}.
		
		(v) 
		Let $\{\ell_1, \ldots, \ell_{p-k}\} = [p] \setminus \fJ$.
		Notice that we have a closed immersion 
		$$
		X \hookrightarrow \Pi_{\fJ}(X) \times_{\kk} \left(\Pi_{\ell_1}(X)\times_\kk \cdots \times_\kk \Pi_{\ell_{p-k}}(X)\right).
		$$
		Since each $\Pi_{\ell_i}(X)$ is zero-dimensional and multiplicity-free, it must be of length $1$ over $\kk$, and so we obtain $\Pi_{\ell_i}(X) \cong \Spec(\kk)$.
		This shows already that $X \cong X'$. 
		Therefore, identifying $\Pi_{\ell_i}(X) \cong \PP_{\kk}^0$, we obtain the equality $P_X(\ttt) =	P_{X'}(t_{j_1},\ldots,t_{j_k})$.
	\end{proof}

	We are now ready for the main result of this section. 
	
	\begin{theorem}
		\label{thm_hsupp_polymatroid}
		Let $X \subset \PP=\PP_\kk^{m_1} \times_{\kk} \cdots \times_{\kk} \PP_\kk^{m_p}$ be a multiplicity-free  variety. 
		Then $\hsupp_\PP(X)$ is a g-polymatroid.
	\end{theorem}
	\begin{proof} 
		Set $\mathscr{C}:=\hsupp_\PP(X)$. We show that $\C$ is a cave (see \autoref{def_caves}) and so it is a g-polymatroid by \autoref{thm_glue_poly}. Let $\bb\in \NN^p$ and consider $\mathscr{A}:=\C_\bb$ the $\bb$-truncation of $\mathscr{C}$ (see \autoref{notation_sec_cave}). 
		By \autoref{prop_mult_free_statements}(i), we have that $\mathscr{A}=\hsupp_\PP(X') + \bb$ for some multiplicity-free variety $X'$ of dimension $\dim(X)-|\bb|$. 
		From \autoref{thm_pos_multdeg}, it follows that $\mathscr{A}^{\ttop}=\msupp_\PP(X') + \bb$ is a polymatroid. 
		 By  \autoref{lem_stalactite_mult_free}, condition (b) in \autoref{def_caves} is satisfied for $\mathscr{A}$. 
		 Since Hilbert supports of zero-dimensional varieties are singletons (hence polymatroids), we can assume by induction on $\dim(X)$ that $\mathscr{A}$ is a g-polymatroid if $\bb\neq \mathbf{0}$, finishing the proof.  
	\end{proof}

\section{Proofs of main results}
	\label{sect_proof_main}

	This section contains the main results of the paper. 
	In particular, we provide the proofs of \autoref{thmA}, \autoref{thmB} and \autoref{thmC}.
	The following setup is used throughout this section.

\begin{setup}
	\label{setup_main_results}
	Let $S=\kk[x_{1,0},\ldots,x_{1,m_1}]\otimes_\kk \cdots \otimes_\kk \kk[x_{p,0},\ldots,x_{p,m_p}]$ be a standard $\NN^p$-graded polynomial ring.
	We see $S$ as the multihomogeneous coordinate ring of $\PP:=\PP_\kk^{m_1} \times_\kk \cdots \times_\kk \PP_\kk^{m_p} = \multProj(S)$.
	Let $\nn \subset S$ be the multigraded irrelevant ideal $\nn = \left(x_{1,0},\ldots,x_{1,m_1}\right) \cap \cdots \cap \left(x_{p,0},\ldots,x_{p,m_p}\right)$ of $S$.
\end{setup}

\subsection{Results on multiplicity-free varieties}

Before proving our main results we need some technical lemmas. 
The first one makes a translation from the twisted $K$-polynomial to a suitable Hilbert polynomial.

\begin{lemma}
	\label{lem_K_poly_reduce_Hilb}
	Let $M$ be a finitely generated $\ZZ^p$-graded $S$-module, and consider the twisted $K$-polynomial $\K(M;\zzz) = \sum_\bn c_\bn(M) \,\zzz^\bn$ of $M$.
	Then, there exists a standard $\NN^p$-graded  polynomial extension $S'$ of $S$, $S' = \kk[x_{1,0},\ldots,x_{1,m_1'}]\otimes_\kk \cdots \otimes_\kk \kk[x_{p,0},\ldots,x_{p,m_p'}]$, such that for the $\ZZ^p$-graded $S'$-module $M':=M\otimes_S S'$ we have 
	$$
	P_{M'}(\ttt)= \sum_{\bn \in \NN^p } c_{\bn}(M) \binom{t_1+m_1'-n_1}{m_1'-n_1} \cdots\binom{t_p+m_p'-n_p}{m_p'-n_p}.
	$$
\end{lemma}
\begin{proof}
	The $K$-polynomials of $M$ and $M'$ coincide because their multigraded Betti numbers are the same.	
	We have the following equation
	$$
	\Hilb_{M'}(\ttt) \;=\; \frac{\K(M; \ttt)}{(1-t_1)^{m_1'+1}\cdots (1-t_p)^{m_p'+1}}.
	$$
	Notice that $\K(M;\ttt) = \sum_{\bn =(n_1,\ldots, n_p )} c_{\bn}(M)(1-t_1)^{n_1}\cdots (1-t_p)^{n_p}$, and thus we obtain
	$$
	\Hilb_{M'}(\ttt) \;=\; \sum_{\bn \in \NN^p} \frac{c_\bn(M)}{(1-t_1)^{m_1'+1-n_1}\cdots(1-t_p)^{m_p'+1-n_p}}.
	$$
	The number of $c_\bn(M)$ that are nonzero is finite, and so we may choose $m_i' \ge n_i$ for all $1 \le i \le p$ and $\bn = (n_1,\ldots,n_p) \in \NN^p$ such that $c_\bn(M) \neq 0$.
	Hence we may expand the above Hilbert series and obtain the identity below
	$$
	\Hilb_{M'}(\ttt) \;=\; \sum_{\bn \in \NN^p} c_\bn(M) \left(\sum_{v_1,\ldots,v_p \ge 0} \binom{v_1+m_1'-n_1}{m_1'-n_1} \cdots \binom{v_p+m_p-n_p}{m_p'-n_p} t_1^{v_1}\cdots t_p^{v_p}\right).
	$$
	The claimed formula for $P_{M'}(\ttt)$ follows from this.
\end{proof}

The next lemma shows that, in the case of multiplicity-free varieties, the above translation from a $K$-polynomial to a Hilbert polynomial always holds. 

\begin{lemma}
	\label{lem_Hilb_poly_mult_free}
	Let $X \subset \PP=\PP_\kk^{m_1} \times_{\kk} \cdots \times_{\kk} \PP_\kk^{m_p}$ be a multiplicity-free variety.
	Consider the twisted $K$-polynomial $\K(X;\zzz) = \sum_\bn c_\bn(X) \,\zzz^\bn$ of $X$.
	Then, for any $\bn = (n_1,\ldots,n_p) \in \NN^p$, we have that $c_\bn(X) = 0$ whenever $n_i \ge  m_i+1$ for some $i \in [p]$.
	As a consequence, we obtain the equality  
	$$
	P_{X}(\ttt) \;=\; \sum_{\bn \in \NN^p } c_{\bn}(X) \binom{t_1+m_1-n_1}{m_1-n_1} \cdots\binom{t_p+m_p-n_p}{m_p-n_p}.
	$$
\end{lemma}
\begin{proof}
	After possibly extending the field, we can assume $\kk$ is infinite. 
	We proceed as in \autoref{lem_eq_Hilb_pol_funct_mult_free}.
	Let $\fP \subset S$ be the $\NN^p$-graded prime ideal corresponding to $X$.
	We consider a generic initial ideal $J :=\gin_>(\fP) \subset S$, then  $J$ is square-free due to \autoref{thm_gin_mult_free}. 
	Similarly to \autoref{lem_general_IE}, we can compute the $K$-polynomial of $S/J$ in terms of the $K$-polynomials of sums of the minimal primes of $J$.
	Notice that $\K(S/\pp_\ba;\zzz) = z_1^{a_1}\cdots z_p^{a_p}$ for any $\ba = (a_1,\ldots,a_p) \in \NN^P$ and that $\pp_{\ba_{1}} + \cdots + \pp_{\ba_{k}} = \pp_{\max\{\ba_{1}, \ldots, \ba_{k}\}}$ for any $\ba_{1}, \ldots, \ba_{k} \in \NN^p$ (see \autoref{rem_mult_free_varieties} for notation).
	Therefore the first statement follows from the fact that the minimal primes of $J$ are of the form $\pp_\ba$ with $\ba = (a_1,\ldots,a_p) \in \NN^p$ and $a_i \le m_i$ (see \autoref{rem_corr_prime_mult}).
	Having the first statement we obtain the second one by following the  proof of \autoref{lem_K_poly_reduce_Hilb}.
\end{proof}

\begin{corollary}\label{cor_transl_hilb_K}
Under the assumptions and notations of \autoref{lem_Hilb_poly_mult_free}, we have that $\supp(\K(X;\zzz))$ and  $\hsupp_\PP(X)$ are related by the equality
\begin{equation*}
	\supp(\K(X;\zzz)) = (m_1,\ldots, m_p)-\hsupp_\PP(X).
\end{equation*} 
\end{corollary}

\begin{example} \label{exam_main_sec_1} 
		Let $I_w\subset S$ be the ideal in our running example \autoref{sec_running_example} and $X=\multProj(S/I_w) \subset \big(\PP_\kk^4\big)^5$. 
		By \autoref{cor_transl_hilb_K} we have $\supp(\K(X;\zzz))=(4,4,4,4,4)-\hsupp_\PP(X)$. Then, from \autoref{exam_mult_free_sec} we obtain:
	\begin{equation*}\label{eq:K-example}
		\K(X;\zzz)=\zzz^{310}+\zzz^{301}+\zzz^{220}+\zzz^{211}+\zzz^{130}+\zzz^{121}+\zzz^{031}-\zzz^{320}-2\zzz^{311}-\zzz^{230}-2\zzz^{221}-2\zzz^{131}+\zzz^{321}+\zzz^{231},
	\end{equation*}
	where $\zzz^{n_1n_2n_3} := z_1^{n_1}z_2^{n_2}z_3^{n_3}$.
\end{example}

For the sake of completeness, we provide the following multiprojective version of \cite[Exercise II.5.14]{HARTSHORNE}.

\begin{lemma}
	\label{lem_equal_sect_ring}
	Let $X \subset \PP=\PP_\kk^{m_1} \times_{\kk} \cdots \times_{\kk} \PP_\kk^{m_p}$ be a variety and $\fP \subset S$ be the corresponding $\NN^p$-graded prime ideal. 
	If $R = S/\fP$ is a normal domain, then we have the equality 
	$$
	R \;=\; \bigoplus_{(v_1,\ldots,v_p) \in \NN^p} \, \HH^0\left(X, \OO_X(v_1,\ldots,v_p)\right).
	$$
\end{lemma}
\begin{proof}
	Let $R' = \bigoplus_{(v_1,\ldots,v_p) \in \NN^p} \, \HH^0\left(X, \OO_X(v_1,\ldots,v_p)\right)$.
	We have that $R \subset R' \subset \Quot(R)$.
	Thus to complete the proof it suffices to prove that $R'$ is a finitely generated $R$-module;
	we proceed to show that.
	
	From \cite[Corollary 1.5]{HYRY_MULTGRAD}, we have the short exact sequence 
	$$
	0 \rightarrow R \rightarrow \bigoplus_{(v_1,\ldots,v_p) \in \ZZ^p} \HH^0\left(X, \OO_X(v_1,\ldots,v_p)\right) \rightarrow \HH_\nn^1(R) \rightarrow 0.
	$$
	It is then enough to show that $\big[\HH_\nn^1(R)\big]_{\ge \mathbf{0}} = \bigoplus_{(v_1,\ldots,v_p) \in \NN^p} \left[\HH_\nn^1(R)\right]_{(v_1,\ldots,v_p)}$ is a finitely generated $S$-module.
	Let $F_\bullet : \cdots \rightarrow F_1 \rightarrow F_0 \rightarrow R \rightarrow 0$ be an $\NN^p$-graded free $S$-resolution of $R$ by modules of finite rank.
	Let $\mathcal{C}_\nn^\bullet$ be the \v{C}ech complex corresponding to a minimal set of generators of the ideal $\nn$.
	We consider the double complex $F_\bullet \otimes_S \mathcal{C}_\nn^\bullet$. 
	By analyzing the corresponding spectral sequences, we obtain a converging spectral sequence 
	$$
	E_1^{-p,q} = \HH_\nn^q(F_p) \Longrightarrow \HH_\nn^{q-p}(R).
	$$		
	By \cite[Proposition 3.11]{SAT_SPEC_FIB}, we have the description 
	$$
	\HH_\nn^j(S) \;=\; \bigoplus_{\substack{\fJ \subseteq [p] \\
			\langle\fJ\rangle +1= j}} \left(\bigotimes_{i \in \fJ} \frac{1}{\bx_i}\kk[\bx_i^{-1}]\right) \otimes_\kk \left(\bigotimes_{i \not\in \fJ}\kk[\bx_i]\right)
	$$
	where $\bx_i = \lbrace x_{i,0},\ldots,x_{i,m_i} \rbrace$ and $\langle\fJ\rangle = \sum_{i \in \fJ} m_i$.
	This shows that $\big[\HH_\nn^j(S)\big]_{\ge \mathbf{0}}$ is a finitely generated $S$-module for all $j \ge 0$.
	Consequently, we obtain that $\big[\HH_\nn^j(R)\big]_{\ge \mathbf{0}}$ is a finitely generated $S$-module for all $j \ge 0$, hence completing the proof.
\end{proof}

We now introduce some necessary notation to enunciate our main theorem. 
This notation is inspired by a known description of Hilbert series in the single graded setting (see \cite[Theorem 4.4.3]{BRUNS_HERZOG}).

\begin{notation}
	\label{nota_bounds}
	Let $X \subset \PP=\PP_\kk^{m_1} \times_{\kk} \cdots \times_{\kk} \PP_\kk^{m_p}$ be a variety. 
	For any nonempty subset $\fJ = \{j_1,\ldots,j_k\} \subsetneq [p]$, we define the bigraded section ring
	$$
	\mathcal{S}_{\fJ}(X) \;:=\; \bigoplus_{u,v \in \NN} \left[\mathcal{S}_\fJ(X)\right]_{(u,v)} \quad \text{ where }\quad  \left[\mathcal{S}_\fJ(X)\right]_{(u,v)} \;:=\; \bigoplus_{\substack{(v_1,\ldots,v_p) \in \NN^p\\ \sum_{j\in \fJ}v_j =u, \;\; \sum_{j\notin \fJ}v_j =v }} \HH^0\left(X, \OO_X(v_1,\ldots,v_p)\right).
	$$
	When $\fJ = [p]$, we define the graded section ring 
	$$
	\mathcal{S}(X)  \;:=\; \bigoplus_{u \in \NN} \left[\mathcal{S}(X)\right]_{u} \quad\text{ where }\quad \left[\mathcal{S}(X)\right]_{u} \;:=\; \bigoplus_{\substack{(v_1,\ldots,v_p) \in \NN^p\\ \sum_{j\in [p]}v_j=u }} \HH^0\left(X, \OO_X(v_1,\ldots,v_p)\right).
	$$
	Then, for any nonempty subset $\fJ \subseteq [p]$, we define the following invariant
	$$
	b_\fJ(X, \PP) \;:= \; \begin{cases}
		\codim(X, \PP) + \reg\left(\mathcal{S}(X)\right) & \text{ if } \fJ = [p]\\
		\sum_{j \in \fJ} (m_j+1) + \max\big\lbrace u \in \ZZ \mid P_{\mathcal{S}_\fJ(X)}(u,v) \neq h_{\mathcal{S}_\fJ(X)}(u,v) \text{ for } v \gg 0\big\rbrace  & \text{ if } \fJ \subsetneq [p].
	\end{cases}
	$$
	Here $\reg\left(\mathcal{S}(X)\right)$ denotes the Castelnuovo-Mumford regularity of the graded $\kk$-algebra $\mathcal{S}(X)$.
	We also define the codimension relative to the subset $\fJ$ as 
	$$
	\codim_\fJ\left(X, \PP\right) \;:=\; \codim\left(\Pi_\fJ(X), \Pi_\fJ(\PP)\right),
	$$
	where $\codim\left(\Pi_\fJ(X), \Pi_\fJ(\PP)\right)$ denotes the codimension of $\Pi_\fJ(X)$ in $\Pi_\fJ(\PP) = \PP_\kk^{m_{j_1}} \times_\kk \cdots \times_\kk \PP_\kk^{m_{j_k}}$.
\end{notation}

In the case of multiplicity-free varieties, the next proposition gives a formula for the degree of the twisted $K$-polynomial in terms of any subset of variables. 

\begin{proposition}
	\label{prop_deg_twisted_K_poly}
	Let $X \subset \PP=\PP_\kk^{m_1} \times_{\kk} \cdots \times_{\kk} \PP_\kk^{m_p}$ be a multiplicity-free variety.
	Consider the twisted $K$-polynomial $\K(X;\zzz) = \sum_\bn c_\bn(X) \,\zzz^\bn$ of $X$.
	Let $\fJ = \{j_1,\ldots,j_k\} \subset [p]$ be a subset.
	Then we have the equality 
	\begin{equation*}
		\deg_{z_{j_1},\ldots,z_{j_k}}\left(\K(X;\zzz)\right) \;=\; b_\fJ(X, \PP),
	\end{equation*}
	where $\deg_{z_{j_1},\ldots,z_{j_k}}\left(\K(X;\zzz)\right)$ denotes the degree of $\K(X;\zzz)$ in terms of the variables $z_{j_1},\ldots,z_{j_k}$.
\end{proposition}
\begin{proof}
	Let $\fP \subset S$ be the $\NN^p$-graded prime ideal corresponding to $X$, and set $R = S/\fP$.		
	By \autoref{lem_equal_sect_ring} and \autoref{thm_gin_mult_free}, we have the equalities 
	$$
	\left[\mathcal{S}(X)\right]_{u} \;=\; \bigoplus_{\substack{(v_1,\ldots,v_p) \in \NN^p\\ \sum_{j\in [p]}v_j=u }} R_{(v_1,\ldots,v_p)}
	\quad \text{ and } \quad 
	\left[\mathcal{S}_\fJ(X)\right]_{(u,v)} \;=\; \bigoplus_{\substack{(v_1,\ldots,v_p) \in \NN^p\\ \sum_{j\in \fJ}v_j =u, \;\; \sum_{j\notin \fJ}v_j =v }} R_{(v_1,\ldots,v_p)}
	$$
	for any $\emptyset \subsetneq \fJ \subsetneq [p]$.

	We first assume $\fJ = [p]$.
	Since $\Hilb_{\mathcal{S}(X)}(t) = \Hilb_R(t,\ldots,t)$, we have an equality $\K(\mathcal{S}(X);z) = \K(X;z,\ldots,z)$ in terms of twisted $K$-polynomials.
	From \autoref{cor_positivity_shellable} and \autoref{thm_gin_mult_free}, we know that the sign of $c_\bn(R)$ changes depending on the parity of $|\bn|$.
	The above substitution is degree preserving, and so we must have $\deg\left(\K(\mathcal{S}(X);z)\right) = \deg\left(\K(X;\zzz)\right)$.
	Since $\mathcal{S}(X)$ is  Cohen-Macaulay (see \autoref{thm_gin_mult_free}), \cite[Theorem 4.4.3]{BRUNS_HERZOG} yields the necessary equality 
	$
	\deg\left(\K(\mathcal{S}(X);z)\right) = \codim(P) + \reg(\mathcal{S}(X)) = b_{[p]}(X, \PP).
	$
	So the case $\fJ = [p]$ is complete. 
	
	We now consider a nonempty subset $\fJ = \{j_1,\ldots,j_k\} \subsetneq [p]$.
	Without any loss of generality, we may assume that $\fJ = \{1, \ldots, k\}$.
	To simplify notation, let 
	$$
	T = \mathcal{S}_\fJ(X), \quad M_1 = -1+ \sum_{j \in \fJ}(m_j+1)  \quad \text{and}  \quad M_2 = -1+\sum_{j \notin \fJ}(m_j+1). 
	$$
	Again, after making substitutions $t_1 = t', \ldots, t_k=t'$ and $t_{k+1}=t'',\ldots,t_p=t''$, we get
	$\Hilb_T(t',t'') = \Hilb_R(t', \ldots,t',t'',\ldots,t'')$, and so we obtain the following equality in terms of twisted $K$-polynomials
	$$
	\K\left(T; z',z''\right) \;=\; \K\left(X; z',\ldots,z',z'',\ldots,z''\right).
	$$
	As the sign of $c_\bn(X)$ alternates depending on $|\bn|$ (see \autoref{cor_positivity_shellable} and \autoref{thm_gin_mult_free}) and the above substitution is degree preserving, it follows that $\deg_{z'}\left(\K(T;z',z'')\right) = \deg_{z_{1},\ldots,z_{k}}\left(\K(X;\zzz)\right)$.
	Due to \autoref{lem_Hilb_poly_mult_free}, we also have that $c_{(n_1,n_2)}(T) = 0$ whenever $n_1 \ge M_1+1$ or $n_2 \ge M_2 +1$.
	After the above reductions, we need to compute $\deg_{z'}\left(\K(T;z',z'')\right)$. 
	We also know that the Hilbert polynomial of $T$ is given by 
	$$
	P_{T}(t_1,t_2) \;=\; \sum_{n_1,n_2 \in \NN} c_{(n_1,n_2)}(T) \binom{t_1+M_1-n_1}{M_1-n_1} \binom{t_2+M_2-n_2}{M_2-n_2}
	$$
	(the proof of this fact follows verbatim the proof of \autoref{lem_K_poly_reduce_Hilb}).
	From this we can conclude that
	$$
	-\left(M_1+1-\deg_{z'}\left(\K(T;z',z'')\right)\right) \;=\;  \max\big\lbrace u \in \ZZ \mid P_T(u,v) \neq h_T(u,v) \;\text{  for } v \gg 0 \big\rbrace,
	$$
	and so the proof is complete.
\end{proof}

Finally, we are ready to state and prove the main result of this paper. 

\begin{theorem}
	\label{main}
	Let $X \subset \PP=\PP_\kk^{m_1} \times_{\kk} \cdots \times_{\kk} \PP_\kk^{m_p}$ be a multiplicity-free variety. 
	Then $\supp\left(  \mathcal{K}(X;\zzz)\right)$ is a g-polymatroid. 
	Furthermore, we have the specific description that $\bn = (n_1,\ldots,n_p) \in \supp(\mathcal{K}(X;\zzz))$ if and only if the following inequality holds 
	$$
	\codim_\fJ\left(X, \PP\right) \;\le\; \sum_{j \in \fJ} n_j \;\le\;  b_\fJ\left(X, \PP\right) 
	$$
	for any nonempty subset $\fJ \subseteq [p]$.
\end{theorem}
\begin{proof}
Translations and reflections of g-polymatroids are g-polymatroids \cite[Theorem 14.2.1]{FRANK}, thus the first statement follows from  \autoref{thm_hsupp_polymatroid} and \autoref{cor_transl_hilb_K}.  
Therefore, by  \cite[Theorem 14.2.8]{FRANK}, we have that $\supp\left(  \mathcal{K}(X;\zzz)\right)$ equals to the set of the integer points of the polytope
\[
\mathscr{G}:=
\left\{\mathbf{y}=(y_1,\ldots, y_p)\in\mathbb{R}^p~\mid~ c(\fJ) \leq \sum_{j\in \fJ} y_j \leq b(\fJ),\;\text{ for all }  \fJ\subseteq [p]\right\},
\]
where $c: 2^{[p]} \rightarrow \ZZ, \,c(\fJ):=\min\{\sum_{j\in \fJ}y_j\mid \mathbf{y}\in \supp\left(  \mathcal{K}(X;\zzz)\right)\}$  and $b: 2^{[p]} \rightarrow \ZZ, \,b(\fJ):=\max\{\sum_{j\in \fJ}y_j\mid \mathbf{y}\in \supp\left(  \mathcal{K}(X;\zzz)\right)\}.$ 
By  \cite[Proposition 3.1]{POSITIVITY} and \autoref{cor_transl_hilb_K}, we have $c(\fJ)=\codim_\fJ\left(X, \PP\right)$. 
Finally, the maximum $b(\fJ)$ is by definition the degree of the twisted $K$-polynomial in the variables $\{z_j \mid j\in \fJ\}$. Then, by \autoref{prop_deg_twisted_K_poly}, the proof is complete.
\end{proof}

\begin{example} \label{exam_main_sec_2} 
	Let $X$ be as in \autoref{exam_main_sec_1}. From the explicit description of $\K(X;\zzz)$ therein we obtain that $\bn=(n_1,n_2,n_3)\in \supp(\K(X;\zzz))$ if and only if the following inequalities hold:
	\begin{align*}
		&4\le n_1+n_2+n_3 \le 6, \qquad
		3\le n_1+n_2 \le 5, \qquad 	1\le n_1+n_3 \le 4,\qquad1	\le n_2+n_3 \le 4, \\
	    &0\le n_1 \le 3, \qquad 0	\le n_2 \le 3, \qquad	0\le n_3 \le 1.
		\end{align*}
We note that  these inequalities coincide with the ones in  \autoref{main}.
\end{example}

\subsection{M\"obius function}\label{mobius_sec}
Let $X\subset \PP  =\PP_\kk^{m_1} \times_{\kk} \cdots \times_{\kk} \PP_\kk^{m_p}$ be a multiplicity-free variety and set $\mathscr{P}:=\msupp_\PP(X)$ which, by \autoref{thm_pos_multdeg}, is a polymatroid. 
Consider 
$
\mathscr{P}_{\le}:=\{ \fu\in\NN^p ~\mid~ \fu \ls \fv \text{ for some }   \fv\in \mathscr{P}\},
$
and the poset  
   $\Gamma:=\mathscr{P}_{\ls }\cup\{\hat{1}\}$ given by the componentwise order $\leq$ 
and such that $\hat{1}$ is larger than any element in $\mathscr{P}_{\ls }$. 
The minimum element of this set is $ \mathbf{0} $. The {\it meet} and {\it join} of $\Gamma$ are defined as 
$$\fu \wedge \fv:=\min\{\fu,\fv\},\quad \text{and,} \quad 
\fu \vee \fv:=
\begin{cases}
	\max\{\fu,\fv\}& \text{if } \max\{\fu,\fv\}\in \Pl\\
	\hat{1}& \text{otherwise}.
\end{cases}$$
 The {\it M\"obius function} of  $\Gamma$ is the function $\mu_\Gamma:\Gamma\times \Gamma \to \ZZ$  uniquely determined by the conditions $\mu_\Gamma(\fu,\fv)=0$ if $\fu\not\le\fv$, $\mu_\Gamma(\fu,\fu)=1$ for every $\fu \in \Gamma$, and
$
\sum_{\fu\leq\fw\leq \fv} \mu_\Gamma(\fu,\fw) = 0,\text{ for every pair }\fu< \fv.
$
We now reframe \autoref{prop_mult_free_statements}(iii) in the context of posets.

\begin{theorem}\label{thm_deg_mobius}
	For any $\bn \in \Pl$, we have the equality $\deg_\PP^\bn(X) = -\mu_\Gamma(\fn,\hat{1}).$ 
\end{theorem}
\begin{proof}
	By \autoref{prop_mult_free_statements}(iii), we have that $\deg_\PP^\bn(X) $ is determined by the same recurrence relation as the M\"obius function.
	We correct the formula with a sign since we are assuming $\deg_\PP^\bn(X)=1$ for every $ \bn\in \msupp_\PP(X) $, whereas $ \mu(\bn,\hat{1})=-1 $ by definition.
\end{proof}

Notice that $\Gamma$ is a \textit{semimodular lattice} \cite[\S 3.3]{stanley2012enumerative}. 
As an application of \autoref{thm_deg_mobius} we obtain another proof of the sign alternation in \autoref{cor_positivity_shellable}:
by \cite[Proposition 3.10.1]{stanley2012enumerative} the M\"obius function of a semimodular lattice, in particular of $\Gamma$, alternates in sign. 

 Let $\mathscr{V}:=(m_1,\ldots, m_p)-\Pl$ and   consider the poset $\Psi=\mathscr{V}\cup \{\hat{0}\}$ with componentwise order and $\hat{0}$ the smallest element. 
 By \autoref{cor_transl_hilb_K}, we obtain the following formula for the twisted $K$-polynomial of a multiplicity-free variety $X \subset \PP = \PP_\kk^{m_1} \times_{\kk} \cdots \times_{\kk} \PP_\kk^{m_p}$.

\begin{theorem}\label{prop:K-poly_mobius}
We have the equality
$
	\mathcal{K}(X; \zzz) = \displaystyle\sum_{\bn\in \mathscr{V}}  -\mu_\Psi(\hat{0},\bn)\zzz^{\bn}.
$
\end{theorem}

As an application of \autoref{prop:K-poly_mobius}, we obtain an alternative proof of the result of \cite[Theorem 3]{knutson2009frobenius} in the case of multiprojective varieties. 

\subsection{Grothendieck polynomials}
\label{subsect_Groth_poly}
Given an integer $p\geq 1$ and $j\in[p-1]$, consider the following operators: for any polynomial  $f \in \mathbb{Z}[z_1,\ldots,z_p]$, we set
\[\partial_j(f)
:=\frac{f(z_1,\ldots,z_p)-f(z_1,\ldots,z_{j-1},z_{j+1},z_j,z_{j+2},\ldots,z_p)}{z_j-z_{j+1}},\]
and 
$
	\overline{\partial}_j(f):=\partial_j((1-z_{j+1})f).
$

The {\it Grothendieck polynomial} $\mathfrak{G}_w \in \ZZ[z_1,\ldots,z_p]$ of a permutation $w\in S_p$ was originally defined by Lascoux and Sch\"utzenberger in 
\cite{LASCOUX_SCHUTZENGERGER}. 
These polynomials can be computed inductively as: if  $w_0=(p, p-1, \ldots,1)\in S_p$ is the longest permutation, then 
$\mathfrak{G}_{w_0}=z_1^{p-1}z_2^{p-2}\cdots z_{p-1}$. 
For  any $w\neq w_0$ and $j\in [p-1]$ with $w(j)<w(j+1)$, the polynomial $\mathfrak{G}_{w}$ is defined as $\overline{\partial}_j \mathfrak{G}_{s_jw} $, 
where $s_j=(j,j+1)$ is the transposition of the positions $j$-th and $(j+1)$-th. 
The  sum of the terms of lowest degree of $\mathfrak{G}_w$ is the {\it Schubert polynomial} of $w$, and it is denoted by $\mathfrak{S}_w \in \ZZ[z_1,\ldots,z_p]$.

In \cite{KNUTSON_MILLER_SCHUBERT}, Knutson and Miller showed that  $ \mathfrak{G}_{w} $ is equal to $ \mathcal{K}(\overline{X}_w; \zzz)$, where $\overline{X}_w$ is the \emph{matrix Schubert variety} of $w\in S_p $. 
If $\overline{X}_w$ is multiplicity-free, the corresponding $\mathfrak{S}_w$ is said to be a {\it zero-one Schubert polynomial} \cite{FINK_MESZAROS_DIZIER}. 
Thus, \autoref{main} implies the following theorem.

\begin{theorem}\label{thm:grothendieck}
	Let $w \in S_p$  be such that $\mathfrak{S}_w$ is a zero-one Schubert polynomial, then $\supp(\mathfrak{G}_{w})$ is a g-polymatroid. 
\end{theorem}

 \autoref{thm:grothendieck} settles a special case of \cite[Conjecture 5.5]{monical2019newton} that every Grothendieck polynomial has SNP.
Permutations with a zero-one Schubert polynomial were classified by a pattern avoidance condition in \cite{FINK_MESZAROS_DIZIER}.
In the table below we show the number of such permutations in $S_p$ for $p\le 12$.\footnote{These computations were done with the PermPy python library developed by Homberger and Pantone http://jaypantone.com/software/permpy/ } 

\medskip

\begin{center}
	\begin{tabular}{ |c|c|c|c|c|c|c|c|c|c|c|c|c|c|} 
		\hline
		p  & 1 & 2 & 3 & 4 & 5 & 6 & 7 & 8 & 9 & 10 & 11 & 12 \\
		\hline
		\#  & 1 & 2 & 6 & 24 & 115 & 605 & 3343 & 19038 & 110809 & 656200 & 3941742 & 23962510 \\
		\hline
	\end{tabular}
\end{center}

\medskip

The interpretation of Grothendieck polynomials as $K$-polynomials together with \autoref{prop:K-poly_mobius} also settles \cite[Conjecutre 1.5]{meszaros2022support}. 
However, a proof of this latter conjecture is also implicit in the work of Knutson \cite{knutson2009frobenius}. 

\subsection{Linear polymatroids and beyond}

\begin{definition}
Let $V_1,\dots,V_p$ be  $\kk$-subspaces of a finite dimensional $\kk$-vector space $V$.
We define a polymatroid $\mathscr{P}$ as 
\[
\NN^p\cap \left\{  (y_1,\dots,y_p)\in\mathbb{R}^p \;\,\mid\;\, \sum_{j\in \fJ} y_j \leq \dim_\kk\left(\sum_{i\in \fJ} V_j\right) \text{ for all $\fJ \subset [p]$},\;\, \sum_{i=1}^p y_i =  \dim_\kk\left(\sum_{i=1}^p V_i\right)\right\}.
\]
A polymatroid that arises in this way is called {\it linear}.
\end{definition}

The following proposition is an extension of \cite[Proposition 5.4]{POSITIVITY}. Here, 
we show that for a given linear  polymatroid,  there exists a  multiplicity-free variety that realizes it as the multidegree support.

\begin{proposition}\label{prop_linear_polymatroid}
	Given a linear 
	polymatroid $\mathscr{P}$ on $[p]$, there is a multiplicity-free variety $X \subset \PP = \PP_\kk^{m_1} \times_\kk \cdots \times_\kk \PP_\kk^{m_p}$ such that $\mathscr{P} = \msupp_\PP(X)$.
\end{proposition}
\begin{proof}
	The proof follows similarly to \cite[Proposition 5.4]{POSITIVITY}.
	Choose a representation of $\mathscr{P}$ given by  a $\kk$-vector space $V$  and $\kk$-vector subspaces $V_1,\ldots,V_p$. 
	Let $S$ be the polynomial ring $S = \Sym(V)[x_0] = \kk[x_0,x_1,\ldots,x_q]$, where $q = \dim_\kk(V)$.
	By using the isomorphism $[S]_1 \cong V \oplus \kk$, we identify each $V_i$ with a $\kk$-subspace $U_i$ of $[S]_1$.
	For each $1 \le i \le p$, let $\{x_{i,1},x_{i,2},\ldots,x_{i,r_i}\} \subset [S]_1$ be a basis of the $\kk$-vector space $U_i$.
	Let $T$ be the $\NN^p$-graded polynomial ring 
	$$
	T := \kk\left[y_{i,j} \mid 1 \le i \le p, \,0 \le j \le r_i,\, \deg(y_{i,j}) = \ee_i\right].
	$$
	Induce an $\NN^p$-grading on $S[t_1,\ldots,t_p]$ given by $\deg(t_i) = \ee_i$ and $\deg(x_j) = \mathbf{0}$.
	Consider the $\NN^p$-graded $\kk$-algebra homomorphism 
	$$
	\varphi = T \rightarrow S[t_1,\ldots,t_p], \qquad 
	\begin{array}{ll}
		y_{i,0} \mapsto  x_0t_i  & \text{ for } \;  1 \le i \le p \\
		y_{i,j} \mapsto  x_{i,j}t_i  & \text{ for } \;  1 \le i \le p,\, 1 \le j \le r_i.
	\end{array}
	$$
	Note that $\fP := \Ker(\varphi) \subset T$ is an $\NN^p$-graded prime ideal. 
	Set $R:= T/\fP$ and $X := \multProj(R)$.
	Notice that $\varphi$ yields a rational map
	$$
	\Phi : \PP_\kk^{q} \,\dashrightarrow\, \PP_\kk^{r_1} \times_\kk \cdots \times_\kk \PP_\kk^{r_p},
	$$
	and that $X \subset \PP_\kk^{r_1} \times_\kk \cdots \times_\kk \PP_\kk^{r_p}$ equals the closure of the image of $\Phi$.
	By either \cite[Theorem~1.1]{li2013images} or  \cite[Theorem~3.9]{conca2019resolution}, we obtain that $X$ is a multiplicity-free variety.
	
	It remains to show the equality $\mathscr{P} = \msupp_\PP(X)$.
	For that we utilize the description given in \autoref{thm_pos_multdeg}.
	By construction, for each $\fJ \subseteq [p]$, we obtain the isomorphism 
	$$
	R_{(\fJ)} \;\cong\; \kk\left[x_{i,j}t_i \mid  i \in \fJ,\, 1 \le j \le r_i\right]\left[x_0t_i\mid i \in \fJ\right] \;\subset\; S[t_1,\ldots,t_p];
	$$
	thus, it is clear that 
	\begin{align*}
		\dim\left(R_{(\fJ)}\right) &= \trdeg_\kk\big(\kk\left[x_{i,j}t_i \mid  i\in \fJ,\, 1 \le j \le r_i\right]\left[x_0t_i\mid i \in \fJ\right]\big)\\
		&= \trdeg_\kk\left(\kk\left[\frac{x_{i,j}}{x_0}\mid  i\in \fJ,\, 1 \le j \le r_i\right]\Big[t_i \mid i \in \fJ\Big]\right)\\
		&= \dim_\kk\left(\sum_{i\in\fJ} U_i\right) + \lvert \fJ \rvert = \dim_\kk\left(\sum_{i\in\fJ} V_i\right) + \lvert \fJ \rvert.
	\end{align*}
	Therefore, we have that $\dim\left(\Pi_{\fJ}(X)\right)= \dim_\kk\left(\sum_{i\in\fJ} V_i\right)$, and so the result follows from \autoref{thm_pos_multdeg}.
\end{proof}

From our methods we obtain the following surprising result for linear polymatroids.

\begin{setup}\label{setup_poset}
Given a polymatroid  $\mathscr{P}\subset \NN^p$, we define the set $\mathscr{P}_\le\subset \NN^p$ and the poset $\Gamma:=\Gamma(\mathscr{P})$ following the same strategy in \autoref{mobius_sec}. As before, $\hat{1}$ denotes the largest element in $\Gamma$. Moreover, we denote by $\hat{0}$ the smallest element of $\Gamma$ and by 
$\mu\text{-supp}(\mathscr{P}):=\left\{\fu \in \mathscr{P}_\leq ~\mid~ \mu_\Gamma(\fu,\hat{1}) \neq 0 \right\}$ its {\it M\"obius support}.
\end{setup} 

\begin{theorem}\label{thm:linear_polymatroid}
Assume \autoref{setup_poset}. 
If  $\mathscr{P}$ is linear, then the M\"obius support  
$\mu\text{\rm-supp}(\mathscr{P})$ 
is a g-polymatroid.
\end{theorem}

\begin{proof}
By  \autoref{prop_linear_polymatroid}, there exists a multiplicity-free variety $X$ such that $\Msupp_\PP(X)=\mathscr{P}$.
Therefore, the result follows by \autoref{thm_deg_mobius} and \autoref{thm_hsupp_polymatroid}.
\end{proof}

We propose the following conjecture whose validity would extend  \autoref{thm:linear_polymatroid} to all polymatroids.

\begin{conjecture}\label{conjecture}
Assume \autoref{setup_poset}. For any {\rm(}base{\rm)} polymatroid $\mathscr{P}$, the M\"obius support $ \mu \text{\rm-supp}(\mathscr{P})$ is a g-polymatroid.
\end{conjecture}

We now show that \autoref{conjecture} is also true for all matroids. First, we briefly review some notions.
A {\it matroid} $\mathscr{M}$ is a  polymatroid such that $[\bu]_i\in\{0,1\}$ for every $\bu\in\mathscr{M}$ and $i\in[p]$.
If we interpret the elements of a matroid $\mathscr{M}$ as subsets of $[p]$, then  $\mathscr{M}_\le$ can be seen as a simplicial complex $\Delta(\mathscr{M})$ with vertex set $[p]$. 
The poset $\Gamma(\mathscr{M})$  is the face poset of the complex (with a greatest element added).
An element $i\in[p]$ is called a \textit{coloop} if it is contained in every facet of $\Delta(\mathscr{M})$.
Given any $\mathfrak{I}\in \mathscr{M}_\le$, we have that $ \text{link}_\mathfrak{I}(\Delta(\mathscr{M})) $ induces a matroid $ \mathscr{M}/\mathfrak{I} $ on the set $ [p]\setminus \mathfrak{I} $, called the \textit{contraction} of $ \mathfrak{I} $.

\begin{theorem}\label{thm_matroid}
	 \autoref{conjecture} holds for matroids.
\end{theorem}
\begin{proof}
	By \cite[Proposition 3.8.8]{stanley2012enumerative}, the value of $ \mu_\Gamma(\hat{0},\hat{1})$ equals the (reduced) Euler characteristic of the complex $\Delta(\mathscr{M})$.
	It is known that the Euler characteristic of the independence complex of a matroid is zero if and only if it has a coloop: 
	this follows from \cite[Corollary 7.2.4]{bjorner1992homology} and the determination of the restrictions sets in \cite[Equation 7.9]{bjorner1992homology}.
	More specifically, if $\mathscr{M}$ is coloopless, then in the last lexicographic basis every element can be exchanged (necessarily by a smaller index), hence they are all internally passive.
	
	Next, we prove that if $ \mathscr{M} $ is coloopless, then so are  all  of its contractions  $ \mathscr{M}/\mathfrak{I}$.  
	Assume by contradiction that there exists an element $i \in [p]\setminus \mathfrak{I}$  such that it is contained in every basis of $ \mathscr{M} $ containing $\mathfrak{I}$. 
	Let $\mathfrak{B}$ be a basis of $ \mathscr{M} $ containing $\mathfrak{I}$ and $i$.
	Since $ \mathscr{M} $ is coloopless, there exists a basis $\mathfrak{B}'$ not including $i$. Now, performing  exchange between $\mathfrak{B}$ and  $\mathfrak{B'}$ at index  $i$,  we can find a basis $\mathfrak{B}''$ containing $\mathfrak{I}$ but not $i$, which is a contradiction. Set $\Gamma_{\mathfrak{I}}:=\Gamma\left(\mathscr{M}/\mathfrak{I}\right)$. 
	Since $\mu_{\Gamma_\mathfrak{I}}(\hat{0}, \hat{1})=\mu_{\Gamma}(\mathfrak{I}, \hat{1})$, we conclude that if $ \mathscr{M} $ is coloopless then $ \mu \text{-supp}(\mathscr{M})=\mathscr{M}_\le$
	and so,  in particular, it is a g-polymatroid.
	
	More generally, if there is a nonempty subset $\mathfrak{A}\subset [p]$ of coloops, then by the same reasoning we have that $ \mu(\mathfrak{I},\hat{1}) = 0 $ for every $\mathfrak{I}\in \mathscr{M}_\le$ that does not contain $\mathfrak{A}$. 
	We obtain that  $\mu_\Gamma(\mathfrak{I},\hat{1}) \neq 0$ if and only if $\mathfrak{I}$ is the union of an independent set of the contraction $ \mathscr{M}/\mathfrak{A} $ and $\mathfrak{A}$. 
	Hence, as a point configuration the set $\mu\text{-supp}(\mathscr{M})$ is equal to the set of indicator vectors of the independence sets of $\mathscr{M}/\mathfrak{A}$ each translated by the vector $\mathbf{1}_\mathfrak{A}=\sum_{i\in \mathfrak{A}}\be_i$.
\end{proof}

\begin{remark}
	\autoref{thm_matroid} shows that the M\"obius support of matroids always has exactly one minimum element.
	This does not necessarily hold for an arbitrary polymatroid as seen in our
     running example~\autoref{sec_running_example}.
\end{remark}

In \cite[Conjecture 22]{HuhLogarithmic}, it was conjectured that the normalization of the homogeneous Grothendieck polynomial is Lorentzian.
Motivated by this, we wonder about the Lorentzian property of the following polynomial: for any polymatroid $\mathscr{P}$, we consider the polynomial
\begin{equation*}\label{eq:lorentzian?}
	g_{\mathscr{P}}(w_0,w_1,\ldots,w_p)\;:=\;\sum_{\fn\in\mathscr{P}_\le} -\mu_\Gamma(\fn,\hat{1})(-w_0)^{r(\mathscr{P})-|\fn|}\www^\fn \in \ZZ[w_0,\ldots,w_p].
\end{equation*}
The sign in $-w_0$ is there to ensure that the coefficients are always nonnegative, since $\mu_\Gamma$ alternates in sign. 
Via a similar construction, the proof of Mason's conjecture in \cite[Theorem 4.14]{HUH_BRANDEN} proceeds by associating to each matroid $ \mathscr{M} $ the polynomial
$
	f_{\mathscr{M}}(w_0,w_1,\ldots,w_p) = \sum_{\fn\in\mathscr{M}_\le}w_0^{r(\mathscr{M})-|\fn|}\www^\fn\in \ZZ[w_0,\ldots,w_p]
$	
and proving that $f_{\mathscr{M}}$ is Lorentzian. 
We consider the normalization operator $N : \ZZ[w_0,\ldots,w_p] \rightarrow \ZZ[w_0,\ldots,w_p]$ given by $N(w_0^{a_0}\cdots w_p^{a_p}) = \frac{w_0^{a_0}\cdots w_p^{a_p}}{a_0!\cdots a_p!}$.
Since the support of a Lorentzian polynomial is a polymatroid, in light of \autoref{conjecture},  it is natural to ask the following:

\begin{question}
	 Is $N(g_{\mathscr{P}})$  Lorentzian?
\end{question}

\section*{Acknowledgments}
Most of the research included in this article was developed in the Mathematisches Forschungsinstitut Oberwolfach (MFO) while the authors were in residence at the institute in November-December 2022. 
The authors thank MFO for their hospitality and excellent conditions for conducting research. 
F.C. was partially funded by  FONDECYT Grant 1221133.
Y.C.R. was partially funded by an FWO Postdoctoral Fellowship (1220122N). 
F.M. was partially supported by the grants G0F5921N (Odysseus programme) and G023721N from the Research Foundation - Flanders (FWO), and the grant iBOF/23/064 from KU Leuven. 
J.M. was partially funded by NSF Grant DMS \#2001645/2303605.
The authors thank Alejandro Morales and Jos\'e Samper for helpful conversations.

\bibliography{references.bib}

\end{document}